\documentclass[final]{elsarticle}

\usepackage{amssymb,amsfonts,amsmath,mathtools,amsthm,enumerate}
\usepackage{float}
\usepackage{graphicx}
\usepackage{longtable}
\usepackage{url}
\usepackage{hyperref}

\newtheorem{theorem}{Theorem}
\newtheorem{lemma}[theorem]{Lemma}

\makeatletter
\def\ps@pprintTitle{%
\let\@oddhead\@empty
\let\@evenhead\@empty
\def\@oddfoot{\reset@font\hfil\thepage\hfil}
\let\@evenfoot\@oddfoot
}
\makeatother

\AtBeginDocument{{\noindent\small
This is a preprint version of the paper published open access 
in \emph{Journal of Applied Mathematics} 
at \url{https://onlinelibrary.wiley.com/journal/4185}.}
\vspace{9mm}}

\begin{document}

\begin{frontmatter}

\title{A Mathematical and Optimal Control Model for Rabies Transmission Dynamics Among Humans and Dogs with Environmental Effects}

\author[mysecondaryaddress,mymainaddress]{Mfano Charles}\corref{mycorrespondingauthor}
\cortext[mycorrespondingauthor]{Corresponding author.}
\ead{mfanoc@nm-aist.ac.tz}

\author[comymainaddress]{Sayoki G. Mfinanga}
\ead{gsmfinanga@yahoo.com}

\author[mysecondaryaddress]{G.A. Lyakurwa}
\ead{geminpeter.lyakurwa@nm-aist.ac.tz}

\author[pt]{Delfim F. M. Torres}
\ead{delfim@ua.pt}

\author[mysecondaryaddress]{Verdiana G. Masanja}
\ead{verdiana.masanja@nm-aist.ac.tz}

\address[mysecondaryaddress]{School of Computational and Communication Science and Engineering,\\ 
Nelson Mandela African Institution of Science and Technology (NM-AIST), P.O. Box 447, Arusha, Tanzania}

\address[mymainaddress]{Department of ICT and Mathematics, 
College of Business Education (CBE), P.O. BOX 1968, Dar es Salaam, Tanzania}

\address[comymainaddress]{Muhimbili Centre, 
National Institute for Medical Research (NIMR), 
Dar es Salaam, Tanzania}

\address[pt]{Center for Research and Development in Mathematics and Applications (CIDMA),\\
Department of Mathematics, University of Aveiro, 3810-193 Aveiro, Portugal}


\begin{abstract}
This study presents a deterministic model to investigate rabies transmission dynamics, 
incorporating environmental effects and control strategies using optimal control theory. 
Qualitative and quantitative analyses reveal that the disease-free equilibrium is stable 
when the effective reproduction number \(\mathcal{R}_e < 1\), and unstable when \(\mathcal{R}_e > 1\). 
Mesh and contour plots illustrate an inverse relationship between \(\mathcal{R}_e\) and control 
strategies, including dog vaccination, health promotion, and post-exposure treatment. 
Increased intervention reduces transmission, while higher contact rates among dogs 
raise \(\mathcal{R}_e\). Numerical simulations with optimal control confirm the effectiveness 
of integrated strategies. Vaccination and treatment are identified as key 
interventions for achieving rabies elimination within five years.
\end{abstract}

\end{frontmatter}


\section{Introduction}

Rabies is a zoonotic viral disease characterized by nearly 100\% mortality in humans once 
clinical symptoms manifest, rendering it a significant public health threat, particularly 
in resource-limited settings \cite{kumar2023canine,ghai2022continued}. Although rabies 
has been successfully eliminated in parts of the Western Hemisphere through comprehensive 
control programs, legislative action, and sustained vaccination campaigns, the disease remains 
endemic in many parts of the world, particularly in Africa and Asia where the burden 
is disproportionately high \cite{charles2024mathematical, baghi2024debate, fooks2020rabies}. 
In Tanzania, as in many other countries of the Global South, rabies continues to pose a serious 
public health threat. An estimated 1,500 human deaths occur annually, with over 98\% resulting 
from bites by rabid domestic dogs \cite{cleaveland2018proof}. Surveillance data from 2022 
documented more than 26,000 reported bite incidents, with notable clusters in high transmission 
areas such as Songea, Mkuranga, and Kigoma. These persistent hotspots emphasize the need for 
intensified and sustained control interventions \cite{minja2023}.
 
One such intervention a coordinated \textit{One Health} mass dog vaccination campaign conducted 
in Moshi District in 2018 achieved 74.5\% coverage, reaching nearly 30,000 dogs at an estimated 
cost of just US\$1.61 per animal (equivalent to US\$0.067 per person protected). This initiative 
was associated with the complete elimination of reported human and animal rabies cases in the 
district by April 2019. Comparable modelling studies in regions such as Ngorongoro and Serengeti 
corroborate that achieving annual vaccination coverage of at least 70\% is both cost-effective 
and critical for interrupting transmission. Nevertheless, post-exposure prophylaxis (PEP) access 
remains a significant challenge. In Moshi, only 65\% of bite victims sought PEP, while studies 
across northern and southern Tanzania indicate that 20--30\% of exposed individuals failed to initiate 
treatment. Barriers such as the high cost of treatment (approximately US\$100 per full regimen), 
inconsistent availability, and logistical delays hinder effective response. 

The global health community has set the 2030 ambitious 
plan to eradicate dog-mediated rabies  \cite{baghi2024debate,rupprecht2022rabies}. 
However, achieving this objective necessitates addressing significant challenges, 
including limited awareness, vaccine availability, and dog population management 
\cite{akinsulie2024holistic}. Notably, free-roaming or stray dogs, which are 
difficult to vaccinate, can perpetuate transmission cycles even in communities 
where owned dogs are vaccinated \cite{velasco2017successful}. Inadequate healthcare 
services and under-reporting exacerbate the issue. Effectively addressing these 
challenges requires moving beyond vaccination and comprehending the interaction 
between humans, dogs, and their environment by integrating strategies necessary 
to disrupt the transmission cycle between dogs and humans. A pivotal strategy 
is mass dog vaccination, aiming to reach the threshold of 70\% vaccination coverage 
within the dog population in areas where rabies is prevalent 
\cite{odero2022estimates,hampson2015estimating}. Public awareness campaigns 
are also essential for educating communities about rabies 
prevention and responsible pet ownership, thereby enhancing participation in 
vaccination programs. Enhancing healthcare provision, particularly preventive treatment 
after exposure, which is vital for preventing rabies in humans following a bite. 
Additionally, effective dog population management, improved surveillance and 
reporting systems, increased funding, and political commitment are critical 
for sustaining rabies control initiatives and advancing the elimination endeavour 
\cite{changalucha2019need}. For comprehending the dynamics of infectious diseases, 
and its application in rabies research, mathematical models have been instrumental 
in revealing insights into transmission patterns and the impact of control measures   
also emphasize the need for targeted human-dog contact rates, and environmental 
factors such as urban versus rural settings. A nonlinear mathematical model proposed 
by \cite{bornaa2020mathematical} to study the dynamics of rabies infection among dogs 
and in a human population that is exposed to dog bites, revealed that it is possible 
to control rabies disease infection through reducing contacts of humans to infected dogs 
and dogs to infected dogs, increasing immunisation of dogs, and culling of infected dogs.   
In \cite{hudson2019modelling}, the most effective rabies vaccination strategies for domestic 
dog populations with heterogeneous roaming patterns is investigated; in the study 
a rabies-spread model was used to simulate a potential outbreak and evaluate various 
disease control strategies where 27 vaccination strategies were employed based on a 
complete block design of 50\%, 70\%, and 90\% in four sampled population structure.  
Simulation scenarios provided some theoretical evidence that targeting subpopulations 
of dogs for vaccination, based on the roaming behaviour of dogs, is more effective 
and efficient than recommending  70\% vaccination campaigns. 
In \cite{abdulmajid2021analysis}, an analysis of time-delayed rabies model 
in human and dog populations with controls reveals that time delay influences 
the endemicity of rabies; it was also observed that in dog populations, 
vaccination and the birth of puppies are equally effective measures for rabies control. 
The work \cite{hailemichael2022effect} uses the SEIR model (Susceptible, Exposed, Infection, and Recovered)  
to study the effect of vaccination and culling on the dynamics of rabies transmission 
from stray dogs to domestic dogs.  The model was constructed by dividing the dog population 
into two categories: stray dogs and domestic dogs. The results of analytical and numerical 
studies demonstrated that the annual dog birth rate is a critical factor 
in influencing the occurrence of rabies. 

Despite the existing body of study, critical factors such as the presence of multiple 
rabies hosts including free-roaming dogs, domestic dogs, and environmental contributors 
have often been overlooked. This gap highlights the urgent need for more comprehensive 
approaches to rabies control. Motivated by the pressing demand to enhance strategies 
for managing both emerging and re-emerging rabies outbreaks, this study employs mathematical 
modeling to optimize intervention efforts, particularly in high-risk areas. It emphasizes 
the integration of environmental and deterrent factors, which are frequently 
underrepresented in traditional control theory frameworks.

	
\section{Model formulation}
\label{modelFormulation}

A mathematical modeling approach to examine rabies transmission among humans, 
free-ranging, and domestic dogs was put forward by \cite{charles2024mathematical}. 
Here we administer time-dependent control variables to that model 
to assess their efficacy in mitigating the transmission of rabies disease 
in both humans and dogs. The selection of these variables is guided by 
strategic interventions aimed at decreasing the likelihood of rabies infections 
in host population, as well as minimizing the shedding rate of rabies virus 
by dogs within the environment, including both domestic and free-range populations.  
To achieve optimal control over rabies transmission in humans, domesticated, 
and free-roaming dogs, we integrate four distinct time-dependent control 
variables that are meticulously tailored to address the intricacies of this 
complex epidemiological scenario: $u_1(t)$ measures the effect of promoting 
good health practices and management (surveillance and monitoring) 
in reducing the possibility of human and dog infections with rabies; 
$u_2(t)$ represents the control effort due to vaccination of domestic dogs; 
$u_3(t)$ measures the impact of educating communities about rabies related 
risks and preventive measures against exposure (public awareness and education);  
$u_4(t)$ represents treatment efforts for individuals exposed to suspected rabid 
animals through bites or scratches (Post-Exposure Prophylaxis-PEP).  
The model of the transmission dynamics of rabies disease in human, 
free-range dogs, domestic dogs and environment 
is developed based on the following assumptions:
\begin{itemize}
\item The transmission is due to adequate contact between 
the susceptible individual and the infectious agent.
\item All infectious individuals are subject to both natural and disease induced mortality.
\item All non-infectious individual  are subject to natural death.
\item There is reduced chance for domestic dog to get the disease due to protection from human being.
\item The recruitment rate for each population is greater than natural mortality.
\item Each population is homogeneously mixed that is each individual 
(human or dogs)  has equal chance of getting the disease.
\item No Post-exposure Prophylaxis (PEP)or Pre-exposure prophylaxis (PREP) 
is applied for free-range dogs.
\item Environment  is any material or object 
that is contaminated by the viruses that cause rabies disease.
\item All individual are recruited at a constant rate.
\item Upon recovery human being and domestic dog attain temporary immunity.
\item Upon exposed, human being and domestic dog  receive PEP.
\end{itemize}
Based on the model formulation and underlying assumptions, the flow diagram presented 
in Figure~\ref{fig:2} delineates the transmission dynamics of rabies by integrating 
the model's structural assumptions, definitions of state variables Table~\ref{tab_1}, 
and parameter specifications.
\begin{figure}[H]
\centering
\includegraphics[scale=0.87]{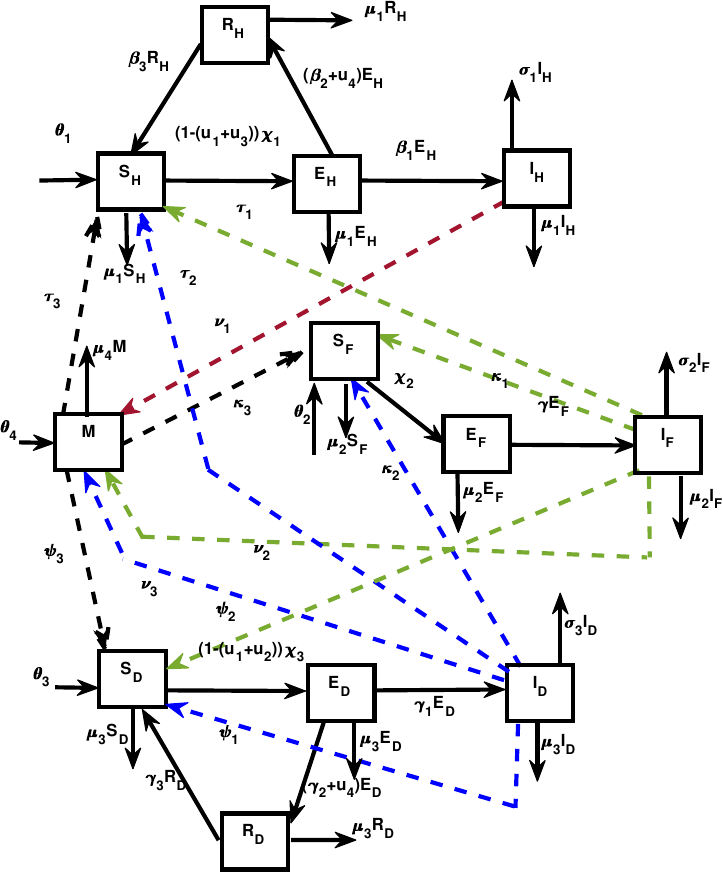}
\caption{\centering Compartmental diagram for rabies transmission with control.}
\label{fig:2}
\end{figure}

\begin{center}
\begin{longtable}{ll}
\caption{\centering Model state variables and their descriptions}\\
\hline 
\textbf{Variable} & \textbf{Description} \\
\hline\hline
\endfirsthead
\multicolumn{2}{c}
{\tablename\ \thetable\ -- \textit{Continued from previous page}} \\
\hline
\textbf{Variable}&\textbf{Description}\\[0.5ex] 
\hline
\endhead
\hline \multicolumn{2}{r}{\textit{Continued on next page}} \\
\endfoot
\hline
\endlastfoot
$S_{H}\left(t\right)$ & susceptible Human at time $t$ \\
$E_{H}\left(t\right)$ &  Exposed Human at time $t$  \\
$I_{H}\left(t\right)$ &  Infected  Human at time $t$  \\
$R_{H}\left(t\right)$ &   Recovered Human at time $t$  \\
$S_{F}\left(t\right)$ &  Susceptible  free- range dog at time $t$  \\
$E_{F}\left(t\right)$ &  Exposed  free-range dog at time $t$  \\
$I_{F}\left(t\right)$ &  Infected  free-range dog at time $t$  \\
$S_{D}\left(t\right)$ &  Susceptible domestic dog at time $t$  \\
$E_{D}\left(t\right)$ &  Exposed domestic dog at time $t$  \\
$I_{D}\left(t\right)$ &  Infected domestic dog at time $t$  \\
$R_{D}\left(t\right)$ &  Recovered domestic dog at time $t$\\
$M\left(t\right)$ &  Number of rabies virus  per unit volume in the  environment  at time $t$  				
\label{tab_1}
\end{longtable}
\end{center}
The dynamics  of  rabies is summarized by the following 
system of ordinary differential equation (ODEs):
\begin{equation}
\left\{
\begin{array}{llll}
\dot{S}_{H}=\theta_{1}+\beta_{3}R_{H}-\mu_{1} S_{H}
-\left(1-\left(u_1+u_3\right)\right)\left(\tau_{1}I_{F} 
+ \tau_{2}I_{D} + \tau_{3} \lambda(M)\right)S_{H},\\
\dot{E}_{H}=\left(1-\left(u_1+u_3\right)\right)
\left(\tau_{1}I_{F} + \tau_{2}I_{D} + \tau_{3} \lambda(M)\right)S_{H}
-\left(\mu_{1}+\beta_{1}+\beta_{2}+u_{4}\right)E_{H},\\
\dot{I}_{H}=\beta_{1}E_{H}-\left(\sigma_{1}+\mu_{1}\right) I_{H},\\
\dot{R}_{H}=\left(\beta_{2}+u_4\right) E_{H}
-\left(\beta_{3}+\mu_{1} \right) R_{H},\\\\
\dot{S}_{F}=\theta_{2}-\left(\kappa_{1}I_{F} + \kappa_{2}I_{D} 
+ \kappa_{3} \lambda(M)\right)S_{F}-\mu_{2}S_{F},\\
\dot{E}_{F}=\left(\kappa_{1}I_{F} + \kappa_{2}I_{D} 
+ \kappa_{3} \lambda(M)\right)S_{F}-\left(\mu_{2}+\gamma\right)E_{F},\\
\dot{I}_{F}=\gamma E_{F}-\left(\mu_{2}+\sigma_{2}\right)I_{F},\\\\
\dot{S}_{D}=\theta_{3}-\mu_{3}S_{D}-\left(1-\left(u_1+u_2\right)\right)
\left(\dfrac{\psi_{1}I_{F}}{1+\rho_{1}} + \dfrac{\psi_{2}I_{D}}{1+\rho_{2}} 
+ \dfrac{\psi_{3}}{1+\rho_{3}}\lambda(M)\right) S_{D}+\gamma_{3}R_{D},\\
\dot{E}_{D}=\left(1-\left(u_1+u_2\right)\right)\left(\dfrac{
\psi_{1}I_{F}}{1+\rho_{1}} + \dfrac{\psi_{2}I_{D}}{1+\rho_{2}} 
+ \dfrac{\psi_{3}}{1+\rho_{3}}\lambda(M)\right) S_{D}
-\left(\mu_{3}+\gamma_{1}+\gamma_{2}+u_{4}\right) E_{D},\\
\dot{I}_{D}=\gamma_{1}E_{D}-\left(\mu_{3}+\sigma_{3}\right) I_{D},\\
\dot{R}_{D}=\left(\gamma_{2}+u_4\right)E_{D}
-\left(\mu_{3}+\gamma_{3}\right)R_{D},\\\\
\dot{M}=\left(\nu_1I_H+\nu_2I_F+\nu_3I_D\right)-\mu_4M,
\end{array}
\right.
\label{eqn79}
\end{equation}
where
\begin{align*}
\lambda(M) = \dfrac{M}{M+C},
\end{align*}
subject to initial non-negative conditions  
\begin{equation}
\begin{aligned}
&S_{H}(0) > 0, \; E_{H}(0) \geq 0, \; I_{H}(0) \geq 0, \; R_{H}(0) \geq 0, 
\; S_{F}(0) > 0, \; E_{F}(0) \geq 0, \; I_{F}(0) \geq 0, \\
&S_{D}(0) \geq 0, \; E_{D}(0) \geq 0, \; I_{D}(0) \geq 0, \; R_{D}(0) \geq 0. 
\end{aligned}
\label{ic}
\end{equation}


\section{Qualitative analysis}
\label{sec:02}

In this section,  we analyze the rabies model \eqref{eqn79} 
by taking consideration of the model's boundedness as well as 
the positivity of the solutions. Stability and the equilibrium 
point are also discussed.

\subsection{Positivity of the solutions and boundedness of the system}
\label{sec:2.2.1}

We begin by proving that the model is well-posed.

\begin{lemma}
All solutions of the system \eqref{eqn79} that  start  in 
$\Omega \subset{\mathbb R }^{12}_{+}$ remain positive for all the time.
\label{theorem:1}
\end{lemma}

\begin{proof}
To establish the existence of a solution for the model given in \eqref{eqn79}, 
we begin by considering the initial conditions in \eqref{ic}. Applying the integral operator 
$\int\limits_{0}^{t} \left(\cdot\right) \, ds$ to each equation in the model \eqref{eqn79}, we derive:
\begin{equation}
\left\{
\begin{array}{llll}
S_{H}\left(t\right)-S_{H}\left(0\right)
=\int\limits_{0}^{t} \left(\theta_{1}+\beta_{3}R_{H}-\mu_{1} S_{H}-\chi_{1}S_{H}\right)ds,\\
E_{H}\left(t\right)-E_{H}\left(0\right)
=\int\limits_{0}^{t}\left(\chi_{1}S_{H}-\left(\mu_{1}+\beta_{1}+\beta_{2}+u_{4}\right)E_{H}\right)ds,\\
I_{H}\left(t\right)-I_{H}\left(0\right)
=\int\limits_{0}^{t}\left(\beta_{1}E_{H}-\left(\sigma_{1}+\mu_{1}\right) I_{H}\right)ds,\\
R_{H}\left(t\right)-R_{H}\left(0\right)
=\int\limits_{0}^{t}\left(\left(\beta_{2}
+u_4\right) E_{H}-\left(\beta_{3}+\mu_{1} \right) R_{H}\right)ds,\\
S_{F}\left(t\right)-S_{F}\left(0\right)
=\int\limits_{0}^{t}\left(\theta_{2}-\chi_{2}S_{F}-\mu_{2}S_{F}\right)ds,\\
E_{F}\left(t\right)-E_{F}\left(0\right)
=\int\limits_{0}^{t}\left(\chi_{2}S_{F}-\left(\mu_{2}+\gamma\right)E_{F}\right)ds,\\
I_{F}\left(t\right)-I_{F}\left(0\right)
=\int\limits_{0}^{t}\left(\gamma E_{F}-\left(\mu_{2}+\sigma_{2}\right)I_{F}\right)ds,\\
S_{D}\left(t\right)-S_{D}\left(0\right)
=\int\limits_{0}^{t}\left(\theta_{3}-\mu_{3}S_{D}-\chi_{3} S_{D}+\gamma_{3}R_{D}\right)ds,\\
E_{D}\left(t\right)-E_{D}\left(0\right)
=\int\limits_{0}^{t}\left(\chi_{3} S_{D}-\left(\mu_{3}+\gamma_{1}+\gamma_{2}+u_{4}\right) E_{D}\right)ds,\\
I_{D}\left(t\right)-I_{D}\left(0\right)
=\int\limits_{0}^{t}\left(\gamma_{1}E_{D}-\left(\mu_{3}+\sigma_{3}\right) I_{D}\right)ds,\\
R_{D}\left(t\right)-R_{D}\left(0\right)
=\int\limits_{0}^{t}\left(\left(\gamma_{2}+u_4\right)E_{D}
-\left(\mu_{3}+\gamma_{3}\right)R_{D}\right)ds,\\
M\left(t\right)-M\left(0\right)
=\int\limits_{0}^{t}\left(\left(\nu_1I_H+\nu_2I_F+\nu_3I_D\right)-\mu_4M\right)ds,
\end{array}
\right.
\label{eqn84}
\end{equation}
where
\begin{align*}
\chi_{1}=\left(1-\left(u_1+u_3\right)\right)
\left(\tau_{1}I_{F} + \tau_{2}I_{D} + \tau_{3} \lambda(M)\right),
\quad \chi_{2}=\left(\kappa_{1}I_{F} + \kappa_{2}I_{D} 
+ \kappa_{3} \lambda(M)\right),\\
\chi_{3}=\left(1-\left(u_1+u_2\right)\right)\left(
\dfrac{\psi_{1}I_{F}}{1+\rho_{1}} + \dfrac{\psi_{2}I_{D}}{1
+\rho_{2}} + \dfrac{\psi_{3}}{1+\rho_{3}}\lambda(M)\right).
\end{align*}
For convenience, from equation \eqref{eqn84} we define the  following functions:
\begin{equation}
\left\{
\begin{array}{llll}
f_1\left(t,\;S_H\right)= \theta_{1}+\beta_{3}R_{H}-\mu_{1} S_{H}-\chi_{1}S_{H},\\
f_2\left(t,\;E_H\right)=\chi_{1}S_{H}-\left(\mu_{1}+\beta_{1}+\beta_{2}+u_{4}\right)E_{H},\\
f_3\left(t,\;E_H\right)=\beta_{1}E_{H}-\left(\sigma_{1}+\mu_{1}\right) I_{H},\\
f_4\left(t,\;R_H\right)=\left(\beta_{2}+u_4\right) E_{H}-\left(\beta_{3}+\mu_{1} \right) R_{H},\\
f_5\left(t,\;S_F\right)=\theta_{2}-\chi_{2}S_{F}-\mu_{2}S_{F},\\
f_6\left(t,\;E_F\right)=\chi_{2}S_{F}-\left(\mu_{2}+\gamma\right)E_{F},\\
f_7\left(t,\;I_F\right)=\gamma E_{F}-\left(\mu_{2}+\sigma_{2}\right)I_{F},\\
f_8\left(t,\;S_D\right)=\theta_{3}-\mu_{3}S_{D}-\chi_{3} S_{D}+\gamma_{3}R_{D},\\
f_9\left(t,\;E_D\right)=\chi_{3} S_{D}-\left(\mu_{3}+\gamma_{1}+\gamma_{2}+u_{4}\right) E_{D},\\
f_{10}\left(t,\;I_D\right)=\gamma_{1}E_{D}-\left(\mu_{3}+\sigma_{3}\right) I_{D},\\
f_{11}\left(t,\;R_D\right)=\left(\gamma_{2}+u_4\right)E_{D}-\left(\mu_{3}+\gamma_{3}\right)R_{D},\\
f_{12}\left(t,\;M\right)=\left(\nu_1I_H+\nu_2I_F+\nu_3I_D\right)-\mu_4M.
\end{array}
\right.
\end{equation} 

Since $\{S_H, E_H, I_H, R_H, S_F, E_F, I_F, S_D, E_D, I_D, R_D, M\} 
\in \Omega \subset \mathbb{R}_{+}^{12},$ and \\
$0 \leq \{S_H, E_H, I_H, R_H, S_F, E_F, I_F, S_D, E_D, I_D, R_D, M\} 
\leq \mathbb{L}_i, \quad i = 1, 2, \ldots, 12$, such that 
\begin{equation}
\begin{split}
\left\Vert f_1(t,\;S_H\left(t\right)) - f_1(t,\;S_H\left(t\right)_{1}) \right\Vert 
&\leq \left\Vert \left(\theta_{1}+\beta_{3}R_{H}\left(t\right)-\mu_{1} 
S_{H}\left(t\right)-\chi_{1}S_{H}\left(t\right)\right) \right. \\
&\qquad - \left. \left(\theta_{1}+\beta_{3}R_{H}\left(t\right)-\mu_{1} 
S_{H}\left(t\right)_{1}-\chi_{1}S_{H}\left(t\right)\right)_{1} \right\Vert\\
&\leq  \mu_{1}\left\Vert  S_{H}\left(t\right)_{1} 
-  S_{H}\left(t\right) \right\Vert\ +\chi_{1}\left\Vert  S_{H}\left(t\right)_{1} 
-  S_{H}\left(t\right) \right\Vert\ \\
&\leq  \left(\mu_{1}+\chi_{1}\right)\left\Vert  S_{H}\left(t\right)_{1} 
-  S_{H}\left(t\right) \right\Vert\ \\ 
&\leq \epsilon_{1}\left\Vert  S_{H}\left(t\right)_{1} 
-  S_{H}\left(t\right) \right\Vert.
\end{split}
\label{eqn85}
\end{equation}
Using a similar approach as in equation \eqref{eqn85}  
for functions $f_i,\;i=2,3, \ldots, 12$, we obtain that
\begin{multline}
\left\{
\begin{aligned}
\left\Vert f_2(t,\;E_H\left(t\right)) - f_2(t,\;E_H\left(t\right)_{1}) 
\right\Vert \leq\epsilon_{2}\left\Vert  E_{H}\left(t\right)_{1} 
-  E_{H}\left(t\right) \right\Vert\ ,\\
\left\Vert f_3(t,\;I_H\left(t\right)) - f_3(t,\;I_H\left(t\right)_{1}) \right\Vert 
\leq\epsilon_{3}\left\Vert  I_{H}\left(t\right)_{1} 
-  I_{H}\left(t\right) \right\Vert\ , \\
\left\Vert f_4(t,\;R_H\left(t\right)) - f_4(t,\;R_H\left(t\right)_{1}) \right\Vert 
\leq\epsilon_{4}\left\Vert  R_{H}\left(t\right)_{1} 
-  R_{H}\left(t\right) \right\Vert\ ,\\
\left\Vert f_5(t,\;S_F\left(t\right)) - f_5(t,\;S_F\left(t\right)_{1}) \right\Vert 
\leq\epsilon_{5}\left\Vert  S_{F}\left(t\right)_{1} 
-  S_{F}\left(t\right) \right\Vert\ , \\
\left\Vert f_6(t,\;E_F\left(t\right)) - f_6(t,\;E_F\left(t\right)_{1}) \right\Vert 
\leq\epsilon_{6}\left\Vert  E_{F}\left(t\right)_{1} 
-  E_{F}\left(t\right) \right\Vert\ , \\
\left\Vert f_7(t,\;I_F\left(t\right)) - f_7(t,\;I_F\left(t\right)_{1}) \right\Vert 
\leq\epsilon_{7}\left\Vert  I_{F}\left(t\right)_{1} 
-  I_{F}\left(t\right) \right\Vert\ , \\
\left\Vert f_8(t,\;S_D\left(t\right)) - f_8(t,\;S_D\left(t\right)_{1}) \right\Vert 
\leq\epsilon_{8}\left\Vert  S_{D}\left(t\right)_{1} 
-  S_{D}\left(t\right) \right\Vert\ , \\
\left\Vert f_9(t,\;E_D\left(t\right)) - f_9(t,\;E_D\left(t\right)_{1}) \right\Vert 
\leq\epsilon_{9}\left\Vert  E_{D}\left(t\right)_{1} 
-  E_{D}\left(t\right) \right\Vert\ , \\
\left\Vert f_{10}(t,\;I_D\left(t\right)) - f_{10}(t,\;I_D\left(t\right)_{1}) \right\Vert 
\leq\epsilon_{10}\left\Vert  I_{D}\left(t\right)_{1} 
-  I_{D}\left(t\right) \right\Vert\ , \\
\left\Vert f_{11}(t,\;S_D\left(t\right)) - f_{11}(t,\;R_D\left(t\right)_{1}) \right\Vert 
\leq\epsilon_{11}\left\Vert  R_{D}\left(t\right)_{1} 
-  R_{D}\left(t\right) \right\Vert\ , \\
\left\Vert f_{12}(t,\;M\left(t\right)) - f_{12}(t,\;M\left(t\right)_{1}) \right\Vert 
\leq\epsilon_{12}\left\Vert  M\left(t\right)_{1} 
-  M\left(t\right) \right\Vert,
\end{aligned}
\right.
\end{multline}
where  $\epsilon_{1}=\mu_{1}+\chi_{1}$, $\epsilon_{2}
=\mu_{1}+\beta_{1}+\beta_{2}+u_{4}$, $\epsilon_{3}=\sigma_{1}+\mu_{1}$,
$\epsilon_{4}=\beta_{3}+\mu_{1}$, $\epsilon_{5}
=\mu_{2}+\chi_{2}$, $\epsilon_{6}=\mu_{2}+\gamma$, $\epsilon_{7}=\mu_{2}+\sigma_{1}$,
$\epsilon_{8}=\mu_{3}+\chi_{3}$, $\epsilon_{9}
=\mu_{3}+\gamma_{1}+\gamma_2+u_4$, $\epsilon_{10}=\mu_{3}+\sigma_{3}$, 
$\epsilon_{11}=\mu_{3}+\gamma_{3}$, $\epsilon_{12}=\mu_{4}$. This satisfy Lipschitz condition 
which guarantees the existence of a solution for the model in \eqref{eqn79}.

\end{proof}

\begin{theorem}
All solutions of system \eqref{eqn79} starting in ${\mathbb{R}}^{12+}$ are uniformly bounded.
\end{theorem}

\begin{proof}
The model system \eqref{eqn79}  can be divided  
in subsections  as  human population, free range,  
and domestic  dogs  as follows:
\begin{equation}
\begin{split}
\dfrac{d\left(S_H+E_H+I_H+R_H\right)}{dt} 
&= \theta_{1}+\beta_{3}R_{H}-\mu_{1} S_{H}-\left(1-\left(u_1+u_3\right)\right)\chi_{1} \\
&\quad +\left(1-\left(u_1+u_3\right)\right)\chi_{1}-\left(\mu_{1}+\beta_{1}+\beta_{2}+u_{4}\right)E_{H}+\beta_{1}E_{H} \\
&\quad -\left(\sigma_{1}+\mu_{1}\right) I_{H}+\left(\beta_{2}+u_4\right) E_{H}-\left(\beta_{3}+\mu_{1} \right) R_{H}.
\end{split}
\label{eqn83}
\end{equation}
Then, equation $\left(\ref{eqn83}\right)$  becomes
\begin{equation}
\dfrac{dN_H}{dt}= \theta_{1}-\left(S_{H}+E_{H}+I_{H}+R_{H}\right)\mu_{1}-\sigma_{1}I_{H}.
\label{eqn82a}
\end{equation}
\begin{equation}
	\dfrac{dN_H}{dt}= \theta_{1}-\left(N_{H} \right)\mu_{1}-\sigma_{1}I_{H}.
	\label{eqn82b}
\end{equation}
We define the integrating  factor as 
\begin{eqnarray}
N_{H}\left(t\right)=e^{\int\limits_{0}^{t} \mu_{1}dt}=e^{\mu_{1} t},
\label{13}
\end{eqnarray}
for $t\rightarrow 0$. Then,  equation $\left(\ref{13} \right)$ is  simplified  as
\begin{eqnarray}
N_{H}(0)\le\frac{\theta_{1}}{\mu_{1}}+Ce^{0} \rightarrowtail N_{H}(0)-\frac{\theta_{1}}{\mu_{1}}\le C.
\label{16}
\end{eqnarray}
A simple manipulation of equation $\left(\ref{16} \right)$ allow us to write that
\begin{equation}
\Omega_{H} = \left\{ \left(S_{H}, E_{H}, I_{H}, R_{H}\right) 
\in \mathbb{R}_{+}^{4} : 0 \leq S_{H}+ E_{H}+I_{H}+R_H \leq \frac{\theta_{1}}{\mu_{1}} \right\}.
\end{equation}
So, using the same  procedure, it can be concluded  that
\begin{align*}
\Omega_{F} &= \left\{ \left(S_{F}, E_{F}, I_{F}\right) \in \mathbb{R}_{+}^{3} : 
0 \leq S_{F}+ E_{F}+ I_{F} \leq \frac{\theta_{2}}{\mu_{2}} \right\}, \\
\Omega_{D} &= \left\{ \left(S_{D}, E_{D}, I_{D}, R_{D}\right) 
\in \mathbb{R}_{+}^{4} : 0 \leq S_{D}+ E_{D}+ I_{D} +R_{D}
\leq \frac{\theta_{3}}{\mu_{3}} \right\}, \\
M\left(t\right) &\leq \Omega_{M} = \max \left\{ \frac{\theta_1 \nu_1}{\mu_1\mu_4}
+\frac{\theta_2 \nu_2}{\mu_2\mu_4}+\frac{\theta_3 \nu_3}{\mu_3\mu_4},M\left(0\right) \right\},
\end{align*}
and the solution is biologically and mathematically meaningfully.
\end{proof}


\subsection{Rabies free equilibrium point $E^{0}$ and the effective reproduction rabies number $\mathcal{R}_e$}
\label{sec:2.2.2}

To obtain the rabies disease free equilibrium point $E^{0}$, 
the left-hand-side of equation 
in the model system  \eqref{eqn79}  is set to zero, such that
\begin{align*}
E^{0} &= \left(
\begin{array}{cccccccccccc}
\dfrac{\theta_{1}}{\mu_{1}}, & 0, & 0, & 0, & \dfrac{\theta_{2}}{\mu_{2}}, 
& 0, & 0, & \dfrac{\theta_{3}}{\mu_{3}}, & 0, & 0, & 0, & 0
\end{array}\right).
\end{align*}
The  $\mathcal{R}_e$ is obtained  by 
\begin{equation}
FV^{-1}=\left[\dfrac{\partial \mathcal{F}_{i}\left(E^{0}\right)}{\partial t}\right]
\left[\dfrac{\partial \mathcal{V}_{i}\left(E^{0}\right)}{\partial t}\right]^{-1},
\end{equation}
where 
\begin{equation*}
{\mathcal{F}}_{i} =
\begin{pmatrix}
\left(1-\left(u_{1}+u_{3}\right)\right) \left(\tau_{1}I_{F}
+\tau_{2}I_{D}+\tau_{3} \lambda \left(M\right)\right)S_{H} \\
0 \\
\left(\kappa_{1}I_{F}+\kappa_{2} I_{D}+\kappa_{3} \lambda \left(M\right)\right)S_{F} \\
0 \\
\left(1-\left(u_{1}+u_{2}\right)\right)\left(\dfrac{\psi_{1}I_{F}}{1+\rho_{1}}
+\dfrac{\psi_{2}I_{D}}{1+\rho_{2}}+\dfrac{\psi_{3}}{1+\rho_{3}}\lambda \left(M\right)\right) S_{D} \\
0 \\
0
\end{pmatrix},\\
\mathcal{V}_{i} =
\begin{pmatrix}
\left(\mu_{1}+\beta_{1}+\beta_{2}+u_{4}\right)E_{H} \\
\left(\sigma_{1}+\mu_{1}\right)I_{H}-\beta_{1}E_{H} \\
\left(\mu_{2}+\gamma\right)E_{F} \\
\left(\mu_{2}+\sigma_{2}\right)I_{F}- \gamma E_{F} \\
\left(\mu_{3}+\gamma_{1}+\gamma_{2}+u_{4}\right) E_{D} \\
\left(\mu_{3}+\delta_{3}\right) I_{D}-\gamma_{1}E_{D} \\
\mu_4M-\left(\nu_1I_H+\nu_2I_F+\nu_3I_D\right)
\end{pmatrix}.
\label{M}
\end{equation*}
The Jacobian matrices  of  $\mathcal{F}_{i}$ and  $\mathcal{V}_{i}$ at $E^{0}$, 
provide  $F$ and $V$ as follows:
\begin{eqnarray}
\frac{\partial\mathcal{F}_{i}}{\partial y_{j}}\big\vert_{E^{0}}=F=
\left(
\begin{array}{ccccccc}
0&0&0&\frac{\left(1-u_{1}-u_{3}\right)\tau_{1}\theta_{1}}{\mu_{1}}
&0&\frac{\left(1-u_{1}-u_{3}\right)\tau_{2}\theta_{1}}{\mu_{1}}&0\cr
0&0&0&0&0&0&0\cr
0&0&0&\frac{\kappa_{1}\theta_{2}}{\mu_{2}}&0&\frac{\kappa_{2}\theta_{2}}{\mu_{2}}&0\cr
0&0&0&0&0&0&0\cr
0&0&0&\frac{\left(1-u_{1}-u_{2}\right)\psi_{1}\theta_{3}}{\left(1+\rho_{1}\right)\mu_{3}}
&0&\frac{\left(1-u_{1}-u_{2}\right)\psi_{2}\theta_{3}}{\left(1+\rho_{2}\right)\mu_{3}}&0\cr
0&0&0&0&0&0&0\cr
0&0&0&0&0&0&0
\end{array}
\right),
\label{job1}
\end{eqnarray} 
\begin{eqnarray}
\frac{\partial\mathcal{V}_{i}}{\partial y_{j}}\big\vert_{E^{0}}
=V=
\setlength{\arraycolsep}{1.5pt}
\left(
\begin{array}{ccccccc}
\mu_{1}+\beta_{1}+\beta_{2}+u_{4}&0&0&0&0&0&0\cr
-\beta_{1}&\sigma_{1}+\mu_{1}&0&0&0&0&0\cr
0&0&\mu_{2}+\gamma&0&0&0&0\cr
0&0&-\gamma &\mu_{2}+\sigma_{2}&0&0&0\cr
0&0&0&0&\mu_3+\gamma_{1}+\gamma_{2}+u_{4}&0&0\cr
0&0&0&0&-\gamma_{1}&\mu_{3}+\sigma_{3}&0\cr
0&-\nu_{1}&0&-\nu_{2}&0&-\nu_{3}&\mu_{4}
\end{array}
\right).
\label{jobi2}
\end{eqnarray} 
Thus, the maximum of the eigenvalues  of \eqref{eqn79} 
gives  the effective  rabies  production number as
\begin{equation}
\mathcal{R}_e= \rho \left(FV^{-1}\right)=\dfrac{R_{33}+R_{21}
+\sqrt{{R_{{21}}}^{2}-2\,R_{{33}}R_{{21}}+4\,R_{{31}}R_{{23}}+{R_{{33}}}^{2}}}{2}
\end{equation}  
with
\begin{align*}
R_{{21}} = \frac{\kappa_{{1}}\theta_{{2}}\gamma}{\mu_{{2}} (\mu_{{2}}+\gamma)(\sigma_{{2}}+\mu_{{2}})}, \;\;
R_{{23}} = \frac{\kappa_{{2}}\theta_{{2}}a_{{3}}}{\mu_{{2}}}, \;\;
a_{3} = {\frac {\gamma}{ \left( \mu_{{3}}+\gamma_{{1}}+\gamma_{{2}}+u_{{4}}\right)  
\left( \sigma_{{3}}+\mu_{{3}} \right) }}, \\
R_{{31}} = \frac{(1-u_{{1}}-u_{{2}})\psi_{{1}}\theta_{{3}}\gamma}{(1
+\rho_{{1}})\mu_{{3}} (\mu_{{2}}+\gamma)(\sigma_{{2}}+\mu_{{2}})}, \;
R_{{33}} = \frac{(1-u_{{1}}-u_{{2}})\psi_{{2}}\theta_{{3}}a_{{3}}}{(1+\rho_{{2}})\mu_{{3}}}.
\end{align*}


\subsection{Global Stability of rabies \(E^{0}\)}
\label{subsec:GS:DFE}

The equilibrium behavior of the model at \(E^{0}\), as defined in equation~\eqref{eqn79}, 
is analyzed using the Metzler matrix framework, following the approaches 
in~\cite{charles2024mathematical,castillo2002computation}.

\begin{theorem}
\label{Th2}	
The rabies  \(E^{0}\) is  globally  
asymptotically stable  when \(\mathcal{R}_{0}<1\) and unstable otherwise.
\end{theorem}

\begin{proof}
Let \(U_{s}\) denote the compartments that do not contribute 
to rabies transmission, while \(U_{i}\) corresponds to those that do. 
If \(G_{2}\) is identified as a Metzler matrix (i.e., all off-diagonal 
elements are non-negative) and \(G_{0}\) possesses only negative real eigenvalues, 
the rabies-free equilibrium is shown to be globally asymptotically stable. 
Consequently, the model represented by equation~\eqref{eqn79} is reformulated as follows:
\begin{equation}
\begin{cases}
\begin{array}{llll}
\dfrac{dU_{s}}{dt} 
&= G_{0}\left(U_{s}-U\left(E^{0}\right)\right)+G_{1}U_{i}, \\
\dfrac{dU_{i}}{dt} 
&= G_{2}U_{i},
\end{array}
\end{cases}
\end{equation}
where
\begin{equation*}
G_{s} - U\left( E^{0}\right)
= \left(
\begin{array}{c}
S_{H} - \dfrac{\theta_{1}}{\mu_{1}} \\[8pt]
R_{H} \\[8pt]
S_{F} - \dfrac{\theta_{2}}{\mu_{2}} \\[8pt]
S_{D} - \dfrac{\theta_{3}}{\mu_{3}} \\[8pt]
R_{D}
\end{array}
\right),
\quad
G_{0} =
\left(
\begin{array}{ccccc}
-\mu & \beta_{3} & 0 & 0 & 0 \\[8pt]
0 & -\left(\beta_{3} + \mu_{1}\right) & 0 & 0 & 0 \\[8pt]
0 & 0 & -\mu_{2} & 0 & 0 \\[8pt]
0 & 0 & 0 & -\mu_{3} & \gamma_{3} \\[8pt]
0 & 0 & 0 & 0 & -\left(\mu_{3} + \gamma_{3}\right),
\end{array}
\right),
\end{equation*}
\begin{equation*}
G_{1}=
\left(
\begin {array}{ccccccc} 
0&0&0&{\dfrac {\tau_{{1}}\theta_{{1}}}{
\mu_{{1}}}}&0&{\dfrac {\tau_{{2}}\theta_{{1}}}{\mu_{{1}}}}&0\\ 
\noalign{\medskip}\beta_{{2}}&0&0&0&0&0&0\cr \noalign{\medskip}0&0&0
&{\frac {\kappa_{{1}}\theta_{{2}}}{\mu_{{2}}}}&0&{\dfrac {\kappa_{{2}}
\theta_{{2}}}{\mu_{{2}}}}&0\cr \noalign{\medskip}0&0&0&{\dfrac {\psi_{{1}}
\theta_{{3}}}{\mu_{{3}} \left( 1+\rho_{{1}} \right) }}&0
&{\dfrac {\psi_{{2}}\theta_{{3}}}{\mu_{{3}} \left( 1+\rho_{{2}} \right) }}
&0\cr \noalign{\medskip}0&0&0&0&\gamma_{{2}}&0&0
\end{array}
\right),
\end{equation*}
\begin{equation*}
\text{and}\;\;G_{2} =	
\left(
\begin {array}{ccccccc} 	
-\mu_{{1}}-\beta_{{1}}-\beta_{{2}}&0&0
&{\frac {\tau_{{1}}\theta_{{1}}}{\mu_{{1}}}}&0&{\frac {\tau_{{2}}
\theta_{{1}}}{\mu_{{1}}}}&0\\ \noalign{\medskip}\beta_{{1}}
&-\sigma_{{1}}-\mu_{{1}}&0&0&0&0&0\\ \noalign{\medskip}0&0
&-\mu_{{2}}-\gamma
&{\frac {\kappa_{{1}}\theta_{{2}}}{\mu_{{2}}}}&0
&{\frac {\kappa_{{1}}\theta_{{2}}}{\mu_{{2}}}}&0\\ 
\noalign{\medskip}0&0&\gamma&-\mu_{{2}}-
\sigma_{{2}}&0&0&0\\ \noalign{\medskip}0&0&0&{\frac {\psi_{{1}}
\theta_{{3}}}{\mu_{{3}} \left( 1+\rho_{{1}} \right) }}
&-\mu_{{3}}-\gamma_{{1}}-\gamma_{{2}}
&{\frac {\psi_{{2}}\theta_{{3}}}{\mu_{{3}} \left( 
1+\rho_{{1}} \right) }}&0\\ \noalign{\medskip}0&0&0
&0&\gamma&-\mu_{{3}}-\sigma_{{3}}&0\\ \noalign{\medskip}0
&\nu_{{1}}&0& \nu_{{2}}&0&\nu_{{3}}&-\mu_{{4}}
\end{array} \right). 
\end{equation*}
Given that the eigenvalues of the matrix \(E_{0}\) are negative 
and the off-diagonal entries of the Metzler matrix \(G_{2}\) 
are non-negative, it follows that the rabies equilibrium 
point \(E^{0}\) is globally asymptotically stable. 
\end{proof}


\subsection{Rabies persistent equilibrium point $E^{*}$} 
\label{sec:2.2.4}

The point
$E^{*}=\left(S^{*}_{H},\;E^{*}_{H},\;I^{*}_{H},\;R^{*}_{H},\;S^{*}_{F},\;
E^{*}_{F}, I^{*}_{F},\;S^{*}_{D},\;E^{*}_{D},\;I^{*}_{D},\;R^{*}_{D},\;M^{*}\right)$,
is obtained by equating equation \eqref{eqn79} to zero:  
\begin{equation*}
\dfrac{dS_{H}}{dt}=\dfrac{dE_{H}}{dt}
=\dfrac{dI_{H}}{dt}=\dfrac{dR_{H}}{dt}=\dfrac{dS_{F}}{dt}
=\dfrac{dE_{F}}{dt}=\dfrac{dI_{F}}{dt}=\dfrac{dS_{D}}{dt}
=\dfrac{dE_{D}}{dt}=\dfrac{dI_{D}}{dt}=\dfrac{dR_{D}}{dt}=0.
\end{equation*}
Upon solving and  considering the  force infection  as  
\begin{align*}
\lambda^{*}_{H} &= \tau_{2}I^{*}_{D} + \tau_{1}I^{*}_{F} + \frac{\tau_{3}M^{*}}{M^{*} + C} ,\;
\lambda^{*}_{F} = \kappa_{1}I^{*}_{F} + \kappa_{2}I^{*}_{D} + \frac{\kappa_{3}M^{*}}{M^{*} + C}, \\
\lambda^{*}_{D} &= \left(1 - u_{1} - u_{2}\right)\left(\frac{\psi_{1}I^{*}_{F}}{1 + \rho_{1}} 
+ \frac{\psi_{2}I^{*}_{D}}{1 + \rho_{2}} + \frac{M^{*}\psi_{3}}{(1 + \rho_{3})(M^{*} + C)}\right),
\end{align*}
the resulting system is
\begin{equation}
\left.
\begin{array}{llll}
S^{*}_{{H}}={\dfrac {\beta_{{3}} \left( \beta_{{2}} \left( \beta_{{2}}+u_{{4}}
\right) E^{*}_{{H}} \right) -\theta_{{1}} \left( \beta_{{3}}+\mu_{{1}}
\right) }{\lambda^{*}_{{H}} \left( u_{{1}}+u_{{3}}-1 \right)  \left( 
\beta_{{3}}+\mu_{{1}} \right) }}
,\;I^{*}_{{H}}={\frac {\beta_{{1}}E^{*}_{{H}}}{\sigma_{{1}}+\mu_{{1}}}},\;
R^{*}_{{H}}={\frac { \left( \beta_{{2}}+u_{{4}} \right) E^{*}_{{H}}}{\beta_{{3
}}+\mu_{{1}}}},\\\\
E^{*}_{{H}}={\dfrac {\theta_{{1}} \left( \beta_{{3}}+\mu_{{1}} \right) }{{
\mu_{{1}}}^{2}+ \left( u_{{4}}+\beta_{{1}}+\beta_{{2}}+\beta_{{3}}
\right) \mu_{{1}}+ \left( u_{{4}}+\beta_{{1}}+\beta_{{2}} \right) 
\beta_{{3}}-\beta_{{2}} \left( \beta_{{2}}+u_{{4}} \right) }},
\end{array}
\right\} \text{for humans}   
\end{equation}
\begin{equation}
\left.
\begin{array}{llll}
S^{*}_{{F}}={\dfrac { \left( -\gamma-\mu_{{2}} \right) E^{*}_{{F}}+\theta_{{2}}
}{\mu_{{2}}}},\;
\lambda^{*}_{{F}}=-{\dfrac {E^{*}_{{F}}\mu_{{2}} \left( \mu_{{2}}+\gamma
\right) }{ \left( \mu_{{2}}+\gamma \right) E^{*}_{{F}}-\theta_{{2}}}}
,\;
I^{*}_{{F}}={\dfrac {\gamma\,E^{*}_{{F}}}{\mu_{{2}}+\sigma_{{2}}}},
\\\\ 
\end{array}
\right\} \text{free-range dogs}   
\end{equation}
\begin{equation}
\left.
\begin{array}{llll}
S^{*}_{{D}}={\dfrac { \left( -u_{{4}}-\mu_{{3}}-\gamma_{{1}}-\gamma_{{2}}
\right) E^{*}_{{D}}+\theta_{{3}}}{\mu_{{3}}}},\;
\lambda^{*}_{{D}}=-{\dfrac {E^{*}_{{D}}\mu_{{3}} \left( \mu_{{3}}+\gamma_{{1}}+
\gamma_{{2}}+u_{{4}} \right) }{ \left( \mu_{{3}}+\gamma_{{1}}+\gamma_{
{2}}+u_{{4}} \right) E^{*}_{{D}}-\theta_{{3}}}},  \\\\
I^{*}_{{D}}={\dfrac {E^{*}_{{D}}\gamma_{{1}}}{\mu_{{3}}+\sigma_{{3}}}}, \; 
R^{*}_{{D}}={\dfrac {E^{*}_{{D}} \left( \gamma_{{2}}+u_{{4}} \right) }{\mu_{{3}
}+\gamma_{{3}}}},\;
M^{*}={\frac {\gamma_{{1}}E^{*}_{{D}}\nu_{{3}}}{\mu_{{4}} \left( \mu_{{3}}+
\sigma_{{3}} \right) }}+{\frac {\beta_{{1}}E^{*}_{{H}}\nu_{{1}}}{\mu_{{4}}
\left( \sigma_{{1}}+\mu_{{1}} \right) }}+{\frac {\gamma\,E^{*}_{{F}}\nu_{
{2}}}{\mu_{{4}} \left( \mu_{{2}}+\sigma_{{2}} \right) }}.
\end{array}
\right\} \text{domestic dogs}   
\end{equation}
Thus, $E^{*}$  persists if 
$E_{H},\; E_{F},\; E_{D} > 0$ and ${\cal R}_e \geq 1$, 
as stated in Theorem~\ref{The}.

\begin{theorem}
The system model \eqref{eqn79} has a unique endemic equilibrium ${E}^*$ 
if $\mathcal{R}_e \geq 1$ and $E_{H},\; E_{F},\; E_{D} > 0$.
\label{The}
\end{theorem}

The endemic equilibrium point $E^{*}$
is globally stable. This is proved in \ref{AppA} 
(see Theorem~\ref{thm:06}).


\subsection{Global Sensitivity Analysis}

In this study, we employed Latin Hypercube Sampling (LHS) and the Partial Rank 
Correlation Coefficient (PRCC) methods to conduct a comprehensive global sensitivity analysis. 
The LHS technique was used to generate 1,000 model parameter runs, incorporating uncertainty 
by treating these parameters as probabilistic variables uniformly distributed over predefined 
ranges. Following this, the PRCC approach was utilized to quantify the degree of monotonic 
relationship between each model input (parameter) and output (state variable), providing 
insights into how each parameter influences the model dynamics. Let \(X_i\) represent 
the sampled input parameter and \(Z_i\) the corresponding model output, with respective 
means \(\bar{X}\) and \(\bar{Z}\). The PRCC value \(r\), which denotes the correlation 
between \(X_i\) and \(Z_i\), is computed as follows:
\begin{equation}
r = \frac{\text{Cov}(X_i, Z_i)}{\sqrt{\text{Var}(X_i)\text{Var}(Z_i)}} 
= \frac{\sum_{i=1}^{N}(X_i - \bar{X})(Z_i - \bar{Z})}{\sqrt{\sum_{i=1}^{N}(X_i 
- \bar{X})^2 \sum_{i=1}^{N}(Z_i - \bar{Z})^2}}.
\end{equation}

The PRCC values range within the closed interval \([-1, +1]\), where values closer 
to either bound indicate a stronger monotonic relationship between input parameters 
and model outputs. A positive PRCC value (\(r > 0\)) reflects a direct (increasing) 
monotonic relationship, suggesting that an increase in the input parameter leads 
to a corresponding increase in the output variable. In contrast, a negative PRCC 
value (\(r < 0\)) denotes an inverse (decreasing) monotonic relationship, whereby 
increasing the input parameter results in a decline in the associated output. 
The results of the sensitivity analysis are visually summarized through histograms 
presented in Figures~\ref{FigRC1}(a)--\ref{FigRC1}(d), illustrating the evolution 
of parameter influence over time. Specifically, Figure~\ref{FigRC1}(a) reveals that 
the parameters $\theta_1$, $\tau_1$, $\tau_2$, $\kappa_1$, and $\kappa_2$ exhibit 
strong positive correlations with the number of infected humans ($I_H$), indicating 
that higher values of these parameters lead to an increased burden of infection 
in human  populations. Moreover, parameters associated with environmental contamination, 
such as the shedding rates $\nu_1$, $\nu_2$, and $\nu_3$, also display positive PRCC values, 
implying their significant role in enhancing indirect rabies transmission through 
environmental exposure. On the other hand, parameters such as $\beta_2$, $\rho_3$, 
$\rho_1$, and $\gamma_2$ exhibit negative PRCC values with respect to $I_H$, 
suggesting that increasing these parameter values contributes to a reduction 
in the number of infections.
\begin{figure}[H]
\begin{minipage}[b]{0.45\textwidth}
\includegraphics[scale=0.55]{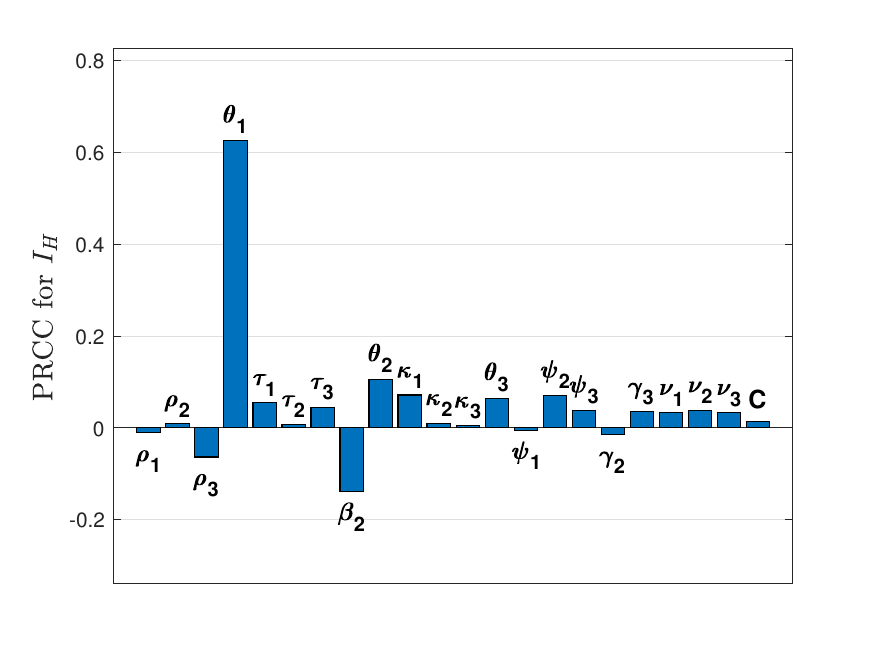}
\centering{(a)}
\end{minipage}
\begin{minipage}[b]{0.45\textwidth}
\includegraphics[scale=0.55]{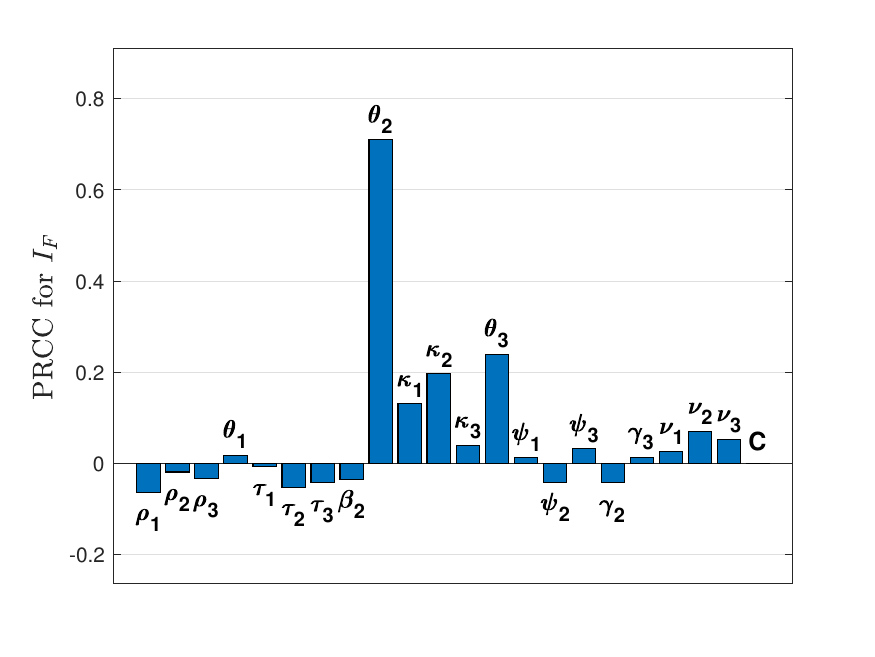}
\centering{(b)}
\end{minipage}
\centering
\caption{\centering The sensitivity analysis of the model outputs 
(populations of infected humans (a), infected free range dogs (b).}
\label{FigRC1}
\end{figure}
Figure~\ref{FigRC1}(b) of infected  free range dogs   $I_F$, identifies $\kappa_1$, 
$\kappa_2$, $\kappa_2$, $\theta_2$, $\theta_3$, $\nu_1$, $\nu_2$, $\nu_3$, 
and $C$ as significant influences of $I_F$ throughout the epidemic.
\begin{figure}[H]
\begin{minipage}[b]{0.45\textwidth}
\includegraphics[scale=0.55]{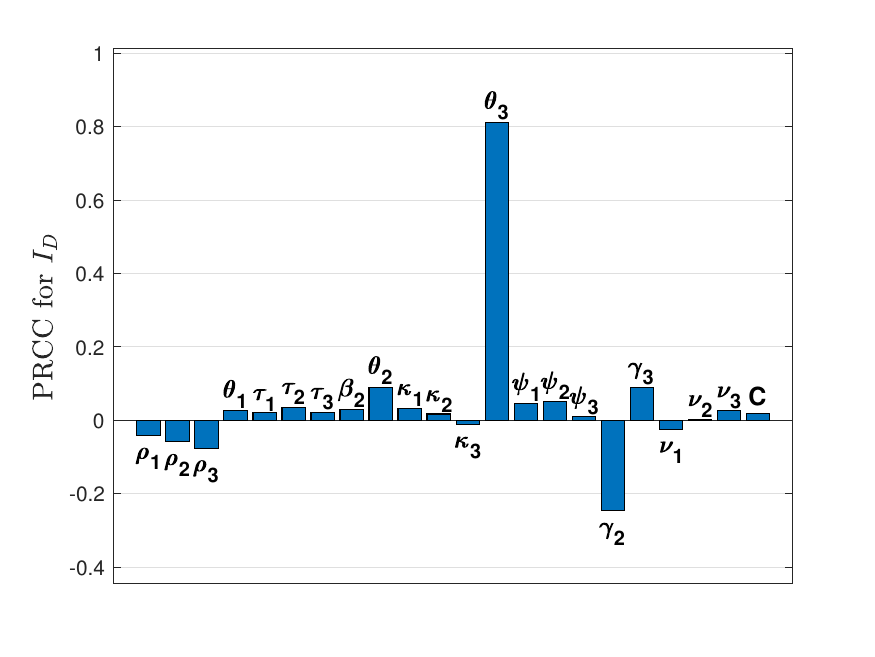}
\centering{(c)}
\end{minipage}
\begin{minipage}[b]{0.45\textwidth}
\includegraphics[scale=0.55]{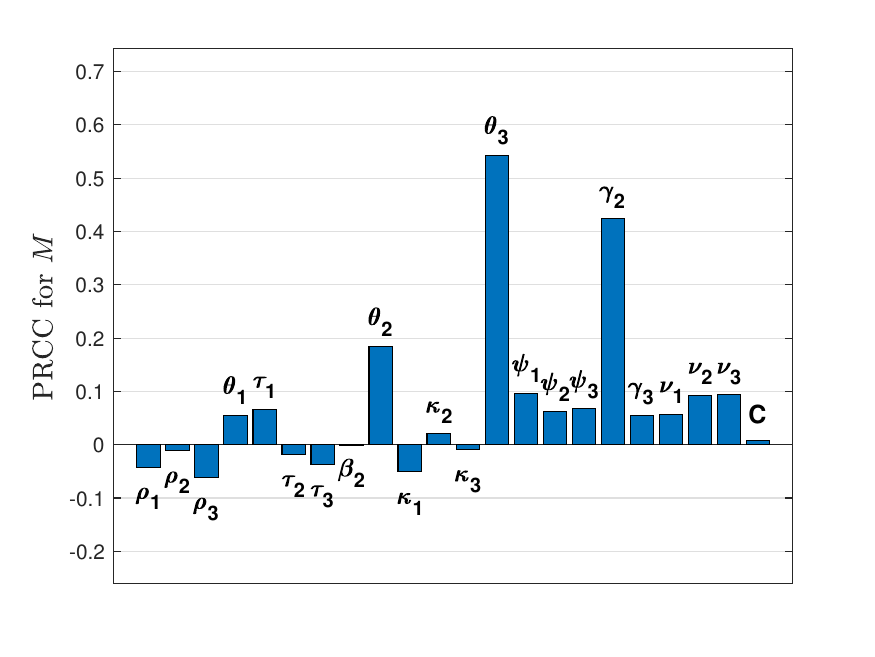}
\centering{(d)}
\end{minipage}
\centering
\caption{\centering The sensitivity analysis of the model outputs 
(populations of infected humans (c), infected free range dogs (d).}
\label{FigRC2}
\end{figure}
Figures~\ref{FigRC2}(c)--\ref{FigRC2}(d) illustrate the sensitivity of the model 
outputs to various input parameters. Specifically, Figure~\ref{FigRC2}(c) shows 
that the parameters $\kappa_1$, $\kappa_2$, $\kappa_3$, $\psi_1$, $\psi_2$, 
$\psi_3$, $\beta_2$, and $\Lambda_2$ exhibit positive PRCC values with respect 
to the number of infected domestic dogs ($I_D$), indicating that increases 
in these parameters contribute to a rise in infection levels within the domestic 
dog population. Conversely, the parameters $\rho_1$, $\rho_2$, and $\rho_3$ 
demonstrate negative PRCC values, suggesting that enhancing these parameters 
may contribute to a reduction in $I_D$.

Furthermore, Figure~\ref{FigRC2}(d) highlights the significant influence 
of the environmental shedding rates $\nu_1$, $\nu_2$, and $\nu_3$ on the 
persistence and amplification of rabies in the environment. These findings 
underscore the critical role of environmental contamination in the indirect 
transmission of the disease. Consequently, effective control strategies 
should prioritize sustained reductions in these key parameters throughout 
the course of the outbreak to mitigate environmental transmission and disease burden.


\subsection{Formulation of the optimal control problem}
\label{sec:3.1}

We aim to reduce the infected human and dog populations while 
also optimizing the costs linked to the implementation of control strategies. 
The objective function that reduces the $I_{H}$  and $I_{D}$ is given by
\begin{equation}
\begin{aligned}
J\left(u_j\right)_{1\leq j\leq 4} = \min_{\left(u_1, u_2, u_3, u_4\right)}
\mathop{{\int_{t_0}^{t_f}}} \left(K_1M + K_2E_H + K_3I_H + K_4E_D 
+ K_5I_D - K_6S_D + \frac{1}{2} \sum_{j=1}^{4} A_ju^{2}_j \right) \, dt
\end{aligned}
\label{eqn80}
\end{equation}
subject to the control system \eqref{eqn79}, where $K_1$ represents the constant 
weight for the environment while $K_2$ and $K_4$ are the constant weights 
for the exposed classes of humans and dogs, respectively. Likewise, $K_3$ and 
$K_5$ are the constant weights for the infectious classes of humans and dogs, respectively, 
and $K_6$ is the constant weight for vaccinated domestic dogs. The coefficients 
$A_1$, $A_2$, $A_3$, and $A_4$ are relative cost weights allocated to each specific 
control measure to transform the integral into the cost occurred over a time frame 
of $t_f$ years. This duration represents the time period for implementing the controls. 
In essence, the cost associated with each control scenario is considered to 
follow a non-linear quadratic pattern, denoted as $\dfrac{A_1u^{2}_1}{2}$. 
This cost reflects the expenses incurred in implementing control strategies aimed 
at promoting environmentally sustainable health practices and management, 
specifically surveillance and monitoring. Similarly, $\dfrac{A_2u^{2}_2}{2}$ 
represents the cost of vaccinating domestic dogs, while $\dfrac{A_3u^{2}_3}{2}$ 
is linked to community education about the risks of rabies exposure and prevention 
(public awareness and education). Meanwhile, $\dfrac{A_4u^{2}_4}{2}$ pertains 
to the gradual treatment of infected individuals
by potentially rabid animals (PEP). Each of these costs contributes to the 
overall expenditure required to address and mitigate the risks associated with rabies.

Considering  $J\left(\text{$u_1, u_2, u_3, u_4$}\right)$, 
the primary goal is to minimize the number of exposed and infected 
for both human and domestic dogs  while simultaneously maximizing 
the number of recovered humans and domestic dogs. 
Consequently, our aim is to determine  
$u^{*}_1$, $u^{*}_2$, $u^{*}_3$, $u^{*}_4$  that satisfy the condition
$J\left(\text{$u^{*}_1, u^{*}_2, u^{*}_3,u^{*}_4$}\right)
= \text{min}  J\left(\text{$u_1, u_2, u_3, u_4$}\right)$,  where 
$U=\left\{u : u \text{ is measurable} \; \text{and}\; 
0\leq u_j\left(t\right) \leq 1, j = 1, \ldots, 4, \; \text{for }\;t\in 
\left[\text{$t_0, t_f$}\right]\right\}$ is the control set,
with the state variables $S_H\left(t\right)$, $E_H\left(t\right)$, $I_H\left(t\right)$,   
$R_H\left(t\right)$, $S_F\left(t\right)$, $E_F\left(t\right)$, $I_F\left(t\right)$,
$S_D\left(t\right)$, $E_D\left(t\right)$, $I_D\left(t\right)$,  
$R_D\left(t\right)$,  and $M\left(t\right)$ satisfying the control system \eqref{eqn79}
and given initial conditions \eqref{ic}.


\subsection{Analysis of the optimal control problem}

This can be proved by characterizing the optimal control via Pontryagin's 
maximum principle, elucidating its essential features and providing 
a comprehensive understanding 
of its structure and dynamics.


\subsubsection{Existence of a unique optimal control}
\label{sec:ex:unq:oc}

The solution to our optimal control problem is ensured by applying the results found  in 
\cite{stoddart1967existence,sagamiko2015optimal}.

\begin{theorem}
The optimal control problem formulated in Section~\ref{sec:3.1} admits a solution.
\end{theorem}

\begin{proof}
We want to show that there exists a control $\bar{u} = \left(\bar{u}_1, \bar{u}_2, 
\bar{u}_3, \bar{u}_4\right) \in U$ such that $J(\bar{u}) \leq J(u)$  for all 
$u = \left(u_1, u_2, u_3, u_4\right) \in U$ subject to \eqref{eqn79}--\eqref{ic}.
This follows from\cite{stoddart1967existence,sagamiko2015optimal} noting that
(i)~the control system \eqref{eqn79} is linear with respect to the control variables;
(ii)~the control set $U$ is  bounded, convex and closed,
and (iii)~the integrand of the objective functional (\ref{eqn80}) is convex in $U$.
\end{proof}

\begin{theorem}
The solution to the optimal control problem formulated in Section~\ref{sec:3.1} is unique.
\end{theorem}

\begin{proof}
Assume that there exists optimal controls $\bar{u}_1, \bar{u}_2, \bar{u}_3, \bar{u}_4$ 
and $\tilde{u}_1, \tilde{u}_2, \tilde{u}_3, \tilde{u}_4$ for the problem
of Section~\ref{sec:3.1}, both achieving 
the minimum value of the objective functional:
\begin{equation}
J(\bar{u}_1, \bar{u}_2, \bar{u}_3, \bar{u}_4) 
= J(\tilde{u}_1, \tilde{u}_2, \tilde{u}_3, \tilde{u}_4) 
= \min_{u_1, u_2, u_3, u_4} J(u_1, u_2, u_3, u_4).
\end{equation}
Consider the difference between the two optimal controls such that:
\begin{equation*}
\delta u_j = \bar{u}_j - \tilde{u}_j \quad \text{for } j=1,2,3,4.
\end{equation*}
Since both controls are optimal, we have
\begin{equation}
J(\tilde{u}_1, \tilde{u}_2, \tilde{u}_3, \tilde{u}_4) 
\leq J(\tilde{u}_1+\delta u_1, \tilde{u}_2
+\delta u_2, \tilde{u}_3+\delta u_3, \tilde{u}_4+\delta u_4).
\label{eqn8}
\end{equation}
Expanding the expression \eqref{eqn8}, and subtracting 
$J(\tilde{u}_1, \tilde{u}_2, \tilde{u}_3, \tilde{u}_4)$ 
from both sides, we get:
\begin{equation}
\left\{
\begin{aligned}
&J(\tilde{u}_1+\delta u_1, \tilde{u}_2+\delta u_2, \tilde{u}_3
+\delta u_3, \tilde{u}_4+\delta u_4) - J(\tilde{u}_1, 
\tilde{u}_2, \tilde{u}_3, \tilde{u}_4) \\
&= \int_{t_0}^{t_f} \left(\sum_{j=1}^{4} 2A_j
\tilde{u}_j\delta u_j + \sum_{j=1}^{4} A_j\delta u_j^{2}\right)dt.
\end{aligned}
\right.
\label{eqn9}
\end{equation}
Since all terms in expression \eqref{eqn9} have non-negative integrand, 
\begin{equation}
0 \leq \int_{t_0}^{t_f} \left(\sum_{j=1}^{4} 2A_j\tilde{u}_j\delta u_j 
+ \sum_{j=1}^{4} A_j\delta u_j^{2}\right)dt.
\end{equation}
Now, since $0 \leq \tilde{u}_j(t) \leq 1$ and $0 \leq \delta u_j(t) \leq 1$ 
for all $t \in [t_0, t_f]$ and $j=1,2,3,4$, we have 
$$
0 \leq 2A_j\tilde{u}_j
\delta u_j \leq 2A_j|\delta u_j|
$$ 
and 
$$
0 \leq A_j\delta u_j^{2} \leq A_j|\delta u_j|, 
$$
where $|\delta u_j|$ represents the absolute value of $\delta u_j$. Therefore, 
\begin{equation}
0 \leq \int_{t_0}^{t_f} \left(\sum_{j=1}^{4} 2A_j\tilde{u}_j\delta u_j 
+ \sum_{j=1}^{4} A_j\delta u_j^{2}\right)dt \leq \int_{t_0}^{t_f} 
\left(\sum_{j=1}^{4} 2A_j|\delta u_j| + \sum_{j=1}^{4} A_j|\delta u_j|\right)dt = 0.
\label{eq1}
\end{equation}
The inequality \eqref{eq1} holds if $\delta u_j(t) = 0$ for all $t \in [t_0, t_f]$ and $j=1,2,3,4$. 
This implies that $\bar{u}_j(t) = \tilde{u}_j(t)$ for all $t \in [t_0, t_f]$ 
and $j=1,2,3,4$, which means the two optimal controls are the same. 
\end{proof}


\subsubsection{Characterization of the optimal control}

Pontryagin's Maximum Principle (PMP), as applied in 
\cite{silva2021optimal,yamada2019comparative,teklu2024impacts}, 
is now employed to characterize the solution in Section~\ref{sec:3.1}. 
The conditions outlined by the PMP transform the optimal control problem 
into the task of minimizing a pointwise function

We introduce the Hamiltonian \(H\) given by
\begin{equation}
\begin{aligned}
& \left.
\begin{array}{l}
H = K_1M + K_2E_H + K_3I_H + K_4E_D + K_5I_D - K_6S_D 
+ \dfrac{1}{2}u^{2}_1A_{1} + \dfrac{1}{2}u^{2}_2A_{2}\\
\quad + \dfrac{1}{2}u^{2}_3A_{3} + \dfrac{1}{2}u^{2}_4A_{3} 
+ \lambda_{1}\left(\theta_{1}+\beta_{3}R_{H}
-\mu_{1} S_{H}-\left(1-\left(u_1+u_3\right)\right)\chi_{1}\right) \\
\quad + \lambda_{2}\left(\left(1-\left(u_1+u_3\right)\right)\chi_{1}
-\left(\mu_{1}+\beta_{1}+\beta_{2}+u_{4}\right)E_{H}\right) 
+ \lambda_{3}\left(\beta_{1}E_{H}-\left(\sigma_{1}+\mu_{1}\right) I_{H}\right) \\
\quad + \lambda_{4}\left(\left(\beta_{2}+u_4\right) E_{H}
-\left(\beta_{3}+\mu_{1} \right) R_{H}\right) 
+ \lambda_{5}\left(\theta_{2}-\chi_{2}-\mu_{2}S_{F}\right) \\
\quad + \lambda_{6}\left(\chi_{2}-\left(\mu_{2}+\gamma\right)
E_{F}\right) + \lambda_{7}\left(\gamma E_{F}-\left(\mu_{2}+\sigma_{2}\right)I_{F}\right)\\
\quad + \lambda_{8}\left(\theta_{3}-\mu_{3}S_{D}-\left(1
-\left(u_1+u_2\right)\right)\chi_{3}+\gamma_{3}R_{D}\right) \\
\quad + \lambda_{9}\left(\left(1-\left(u_1+u_2\right)\right)\chi_{3}
-\left(\mu_{3}+\gamma_{1}+\gamma_{2}+u_{4}\right) E_{D}\right) \\
\quad + \lambda_{10}\left(\gamma_{1}E_{D}-\left(\mu_{3}+\sigma_{3}\right) I_{D}\right) 
+ \lambda_{11}\left(\left(\gamma_{2}+u_4\right)E_{D}-\left(\mu_{3}+\gamma_{3}\right)R_{D}\right) \\
\quad + \lambda_{12}\left(\left(\nu_1I_H+\nu_2I_F+\nu_3I_D\right)-\mu_4M\right),
\end{array}
\right\}
\end{aligned}
\label{eqn81}
\end{equation} 
where
\begin{align*}
\chi_{1} &= \left(\tau_{1}I_{F} + \tau_{2}I_{D} + \tau_{3} \lambda(M)\right)S_{H},\; 
\chi_{2} = \left(\kappa_{1}I_{F} + \kappa_{2}I_{D} + \kappa_{3} \lambda(M)\right)S_{F},\\
\chi_{3} &= \left(\dfrac{\psi_{1}I_{F}}{1+\rho_{1}} + \dfrac{\psi_{2}I_{D}}{1+\rho_{2}} 
+ \dfrac{\psi_{3}}{1+\rho_{3}}\lambda(M)\right) S_{D},\;
\lambda\left(M\right) = \dfrac{M}{M+C},
\end{align*}
and $\lambda_{1}, \ldots, \lambda_{12}$ represent the adjoint variables  
and  $y = \left(M, S_H, E_H, I_H, R_H, S_F, E_F, I_F, S_D,
E_D, I_D, R_D\right)$ are the state variables. 
The following result is a direct consequence
of Pontryagin's maximum principle  \cite{MR186436} 
and our results of Section~\ref{sec:ex:unq:oc}.

\begin{theorem}
\label{thm:PMP}
If the  optimal control to the problem formulated
in Section~\ref{sec:3.1} is $u^{*} = (u^{*}_1, u^{*}_2, u^{*}_3, u^{*}_4)$, 
then there exist corresponding adjoint variables 
$\lambda_{1}, \ldots, \lambda_{12}$ that satisfy the adjoint system
\begin{equation*}
\lambda_{j}'(t)
=-\dfrac{\partial H}{\partial y_j} 
\end{equation*}
subject to the transversality conditions
$$
\lambda_{j}\left(t_f\right)=0, 
$$
for $1\le j\le 12$. Furthermore, the optimal controls 
$u^{*}_i$, $i = 1, \ldots, 4$, are characterized by 
the maximality condition as follows:
\begin{equation}
\left\{
\begin{array}{llll}
u^{*}_1=\max\left\{0,\min\left(1,\;\;\dfrac{\left(\lambda_{{2}} 
-\lambda_{{1}} \right)\left( \tau_{{2}}I_{{D}}+\tau_{{1}}I_{{F}}
+{\dfrac {M\tau_{{3}}}{M+C}} \right) S_{{H}}+\left(\lambda_{{9}}
-\lambda_{{8}}\right) \left( G \right)S_{{D}}}{A_{1}}\right)\right\},\\\\
u^{*}_{2}=\max\left\{0,\min\left(1,\;\;\dfrac{\left(\lambda_{{9}}
-\lambda_{{8}}\right) \left( {\frac {\psi_{{1}}I_{{F}}}{1+\rho_
{{1}}}}+{\dfrac {\psi_{{2}}I_{{D}}}{1+\rho_{{2}}}}+{\dfrac {\psi_{{3}}M
}{ \left( M+C \right)  \left( 1+\rho_{{3}} \right) }} 
\right) S_{{D}}}{A_{2}}\right)\right\},\\\\
u^{*}_{3}=\max\left\{0,\min\left(1,\;\;\dfrac{\left(\lambda_{2}
-\lambda_{1}\right)\left(\tau_{1}I_{F}+\tau_{2}I_{D}
+\dfrac{\tau_{3}M}{M+C}\right)S_{H}}{A_{3}}\right)\right\},\\\\
u^{*}_{4}=\max\left\{0,\min\left(1,\;\;{\dfrac{E_{H}\left(\lambda_{{2}}
-\lambda_{{4}}\right)+E_{{D}}\left(\lambda_{{9}}
-\lambda_{{11}}\right)}{A_{4}}}\right)\right\},
\end{array}
\right.
\end{equation}
where
\begin{equation*}
G={\dfrac {\psi_{{1}}I_{{F}}}{1+\rho_
{{1}}}}+{\dfrac {\psi_{{2}}I_{{D}}}{1+\rho_{{2}}}}+{\dfrac {\psi_{{3}}M
}{ \left( M+C \right)  \left( 1+\rho_{{3}} \right) }}. 
\end{equation*}
\end{theorem}

In Section~\ref{sec:32} we follow \cite{MR4091761}
to implement Theorem~\ref{thm:PMP} 
in \textsf{Matlab} in an algorithm way.


\section{Quantitative analysis}
\label{sec:03}

Before numerical simulations, the parameters of the mathematical model \eqref{eqn79} 
are estimated, to understand the transmission dynamics of the system,
with data from Tanzania.


\subsection{Model fitting and parameter estimation}
\label{sec:31}

To estimate  parameters $\Theta_{i}$, the model is fitted to real-world 
data using optimization techniques. The Nelder--Mead algorithm  was  employed  
to  minimize the difference between the model's predictions $\hat{Y}(\Theta_{i})$  
and the observed data $Y$. The estimated  parameters were obtained  using the least  
square fitting of  $I_{H}\left(t_{i}\right)$ through  discretization of  
the system of ordinary differential equations, 
that is, \eqref{eqn79} without controls, as follows:
\begin{equation}
\begin{aligned}
I_H(t_{i} + \Delta t) 
&= I_H(t_{i}) + \left( \beta_1 \cdot E_H(t_{i}) 
- (\sigma_1 + \mu_1) \cdot I_H(t_{i}) \right) \cdot \Delta t.
\end{aligned}
\end{equation}
The least-square fitting  is to minimize the object function given as
\begin{equation}
\begin{aligned}
\underset{\Theta_{i}}{\min} \, 
\text{MSE}(\Theta_{i}) = \frac{1}{n} \sum_{i=1}^{n} (Y_i - \hat{Y}_i(\Theta_{i}))^2,
\end{aligned}
\label{eqn92}
\end{equation}
where  $\text{MSE}(\Theta_{i})$ represents the Mean Squared Error, 
$n$ is the number of data points, while $Y_i$ and $\hat{Y}_i(\Theta_{i})$ 
are the observed and predicted values, respectively, at time $i$.

We have used data from Tanzania from 1990 to 2018: see Figure~\ref{fig:eghh}.
The  estimated model parameters that we obtained are shown in Table~\ref{T2},
together with some other values taken from the literature.
\begin{figure}
\centering
\includegraphics[scale=0.8]{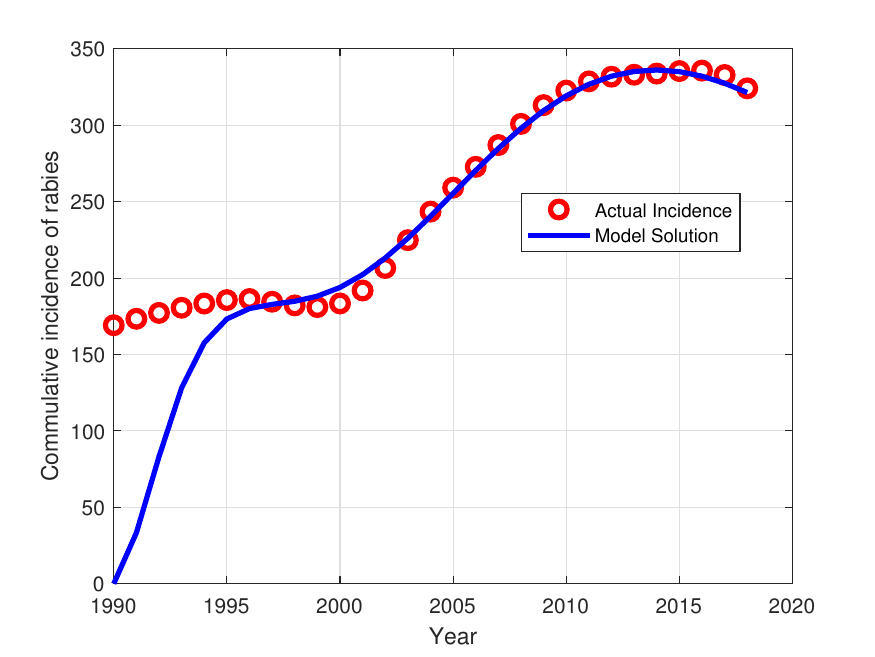}
\caption{\centering The comparison between the reported human rabies 
incidence in Tanzania from 1990 to 2018 and the simulation of 
$I_H\left(t\right)$ from our model \eqref{eqn79}.}
\label{fig:eghh}
\end{figure}

\begin{center}
\begin{longtable}{|l|l|l|l|l|}
\caption{\centering Epidemilogical  parameters data  (Year$^{-1}$) for model  \ref{eqn79}.}\\  \hline 
\textbf{Parameters} & \textbf{Baseline value} & \textbf{Source} 
& \textbf{Estimated value}& $\textbf{Mean}\left(\mu \right) \textbf{and std}\left(\sigma\right) $\\
\hline\hline
\endfirsthead
\multicolumn{4}{c}
{\tablename\ \thetable\ -- \textit{Continued from previous page}} \\
\hline
\textbf{Parameters} & \textbf{Baseline value} & \textbf{Source} 
& \textbf{Estimated value}& $\textbf{Mean}\left(\mu \right) \textbf{and std}\left(\sigma\right) $\\
\hline
\endhead
\hline \multicolumn{4}{r}{\textit{Continued on next page}} \\
\endfoot
\hline
\endlastfoot
$\theta_{1}$ & 2000& (Estimated)&1993.382113 & $\mathcal{ N}\left(1996.691056\;\; 4.4679553 \right)$ \\
$\tau_{1}$ & 0.0004& \cite{tian2018transmission}&0.000405 &$\mathcal{ N}\left(0.000402\;\; 4\times10^{-6} \right)$.\\
$\tau_{2}$ & 0.0004& \cite{tian2018transmission}&0.000604&$\mathcal{ N}\left(0.000502\;\; 1.44\times10^{-4} \right)$.\\
$\tau_{3}$ & $\left[0.0003\;\;  0.0100\right]$ & (Estimated)&0.000303 &$\mathcal{ N}\left(0.000302\;\; 2\times10^{-6} \right)$ .\\
$\beta_{1}$ & $\frac{1}{6}$ & \cite{zhang2011analysis,tian2018transmission}&0.165581& $\mathcal{ N}\left(0.166124\;\; 7.68\times10^{-4} \right)$.\\
$\nu_3$ &0.001& (Estimated)&0.005735&$\mathcal{ N}\left(0.003367\;\; 3.3348\times10^{-3} \right)$.\\
$\beta_{2}$ & $\left[0.54 \;\; 1\right]$ & \cite{zhang2011analysis,abdulmajid2021analysis}&0.540487&$\mathcal{ N}\left(0.5402435\;\;
3.7815\times10^{-4} \right)$.\\
$\beta_{3}$ & 1& (Estimated)&0.999301&$\mathcal{ N}\left(0.9996505\;\; 1.6521\times10^{-4} \right)$.\\
$\mu_1$ & 0.0142& \cite{world2010working,world2013expert}&0.014417&$\mathcal{ N}\left(0.014309\;\; 1.53\times10^{-4} \right)$.  \\
$\sigma_1$ & 1& \cite{zhang2011analysis,abdulmajid2021analysis} &1.006332&$\mathcal{ N}\left(1.03166\;\; 4.47\times10^{-3} \right)$. \\
$\theta_{2}$ & 1000& (Estimated)&1004.12044&$\left(1002.060222\;\; 2.913594\right)$. \\
$\kappa_{1}$ &  0.00006 & (Estimated)&0.000020&$\mathcal{ N}\left(0.000040\;\; 2.8\times10^{-5} \right)$.\\
$\kappa_{2}$ & 0.00005 & (Estimated)&0.000081&$\mathcal{ N}\left(0.000066\;\; 2.2\times10^{-5} \right)$.\\
$\kappa_{3}$ & $\left[0.00001 \;\; 0.00003\right]$&(Estimated)&0.000040&$\mathcal{ N}\left(0.000025\;\; 2.1\times10^{-5} \right)$.\\
$\gamma$ & $\frac{1}{6}$ & \cite{zhang2011analysis,tian2018transmission,abdulmajid2021analysis} & 0.166374
&$\mathcal{ N}\left(0.166520\;\; 2.07\times10^{-4} \right)$.\\
$\nu_1$ &0.001& (Estimated)&0.001958&$\mathcal{ N}\left(0.001479\;\; 6.77\times10^{-4} \right)$.\\
$\sigma_2$ & 0.09 & \cite{zhang2011analysis,addo2012seir}&0.089556&$\mathcal{ N}\left(0.089778\;\; 3.14\times10^{-4} \right)$.\\
$\mu_4$  &0.08& (Estimated)&0.080625&$\mathcal{ N}\left(0.080313\;\; 4.42\times10^{-4} \right)$. \\
$\mu_{2}$ & 0.067 & (Estimated)&0.066268 &$\mathcal{ N}\left(0.066634\;\; 1.58\times10^{-4} \right)$.\\
$\theta_3$ & 1200& (Estimated)&1203.844461&$\mathcal{ N}\left(1201.922230\;\; 2.718444 \right)$.\\
$\psi_{1}$ &0.0004& \cite{hampson2019potential,addo2012seir}&0.000077&$\mathcal{ N}\left(0.000238\;\; 2.28\times10^{-4} \right)$.\\
$\psi_{2}$ & 0.0004 & \cite{hailemichael2022effect}&0.000066&$\mathcal{ N}\left(0.000233\;\; 2.36\times10^{-4} \right)$. \\
$\psi_{3}$ & 0.0003 & (Estimated)&0.000030&$\mathcal{ N}\left(0.0003\;\; 1.91\times10^{-4} \right)$.\\
$\mu_3$ &0.067& (Estimated)&0.080129&$\mathcal{ N}\left(0.073565\;\; 8.056\times10^{-3} \right)$. \\
$\sigma_3$ &0.08& \cite{zhang2011analysis}&0.091393 & $\mathcal{ N}\left(0.085697\;\; 8.056\times10^{-3} \right)$. \\
$\gamma_1$ & $\frac{1}{6}$ & \cite{zhang2011analysis,tian2018transmission}&0.172489&$\mathcal{ N}\left(0.169578\;\; 4.117\times10^{-3} \right)$. \\
$\gamma_2$ & 0.09 & \cite{zhang2011analysis}&0.090308&$\mathcal{ N}\left(0.090154\;\; 2.18\times10^{-4} \right)$.\\
$\gamma_3$ & 0.05&(Estimated)&0.050128&$\mathcal{ N}\left(0.050128\;\; 9.1\times10^{-5} \right)$.\\
$\nu_2$ &0.006& (Estimated)&0.008971&$\mathcal{ N}\left(0.007485\;\; 2.101\times10^{-3} \right)$.\\
$\rho_{1}$ & 10& \cite{ruan2017spatiotemporal}&9.920733&$\mathcal{ N}\left(9.960366\;\; 5.605\times10^{-2} \right)$.\\
$\rho_{2}$ & 8& (Estimated)&8.116421&$\mathcal{ N}\left(8.058211\;\; 8.2322\times10^{-2} \right)$.\\
$\rho_{3}$ & 15 & (Estimated)&14.917005 &$\mathcal{ N}\left(14.958502\;\; 5.8686\times10^{-2} \right)$.\\
$C$  &0.003  (PFU)/mL & (Estimated)&0.003011&$\mathcal{ N}\left(0.003005\;\; 8.0000\times10^{-6} \right)$.
\label{T2}
\end{longtable}
\end{center}


\subsection{Impact of contact rate, Optimal control Strategy and $\mathcal{R}_e$}

Mesh and contour plots are shown in Figures~\ref{Fig12} and \ref{Fig13}.  
They illustrate the relationship between $\mathcal{R}_e$ and $u_1\left(t\right)$,  
which measures the impact of promoting  good health practice and management;
$u_2\left(t\right)$, which represents control effort due to vaccination 
of domestic dogs; and $u_4\left(t\right)$, which  measures the  treatment 
effort given to people bitten or scratched by a potentially rabid animal 
(PREP \& PEP), along with the recovery rate as a result of treating infected individuals. 
It is observed that the magnitude of $\mathcal{ R}_e$ decreases with an increase in 
$u_1\left(t\right)$, $u_2\left(t\right)$, and $u_4\left(t\right)$. This observation 
highlights the effectiveness of $u_1\left(t\right)$, $u_2\left(t\right)$,  
and $u_4\left(t\right)$  interventions in reducing the transmission 
of rabies  across the community.
\begin{figure}[H]
\begin{minipage}[b]{0.45\textwidth}
\includegraphics[height=5.0cm, width=7.5cm]{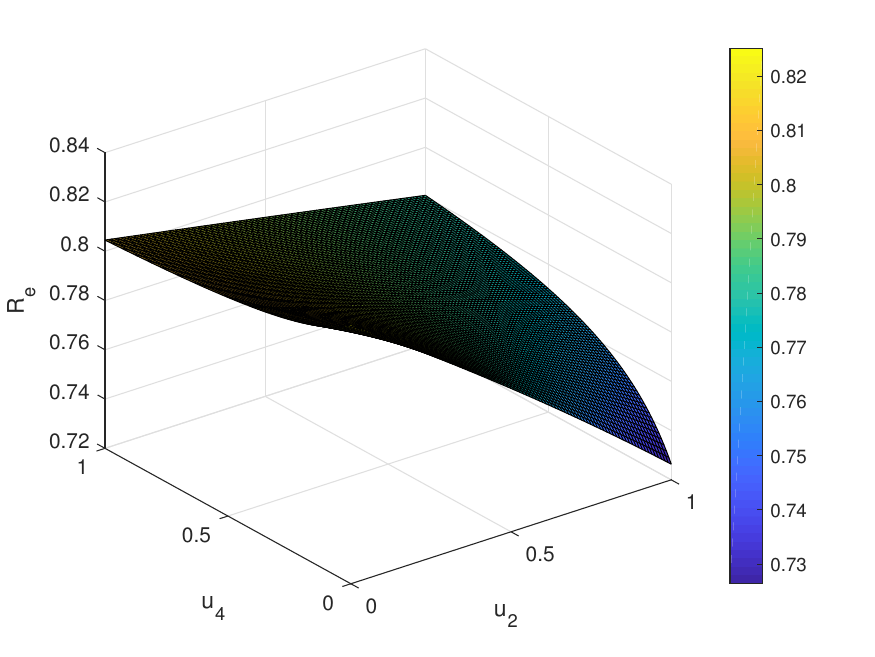}
\centering{(a)}
\end{minipage}
\begin{minipage}[b]{0.45\textwidth}
\includegraphics[height=5.0cm, width=7.5cm]{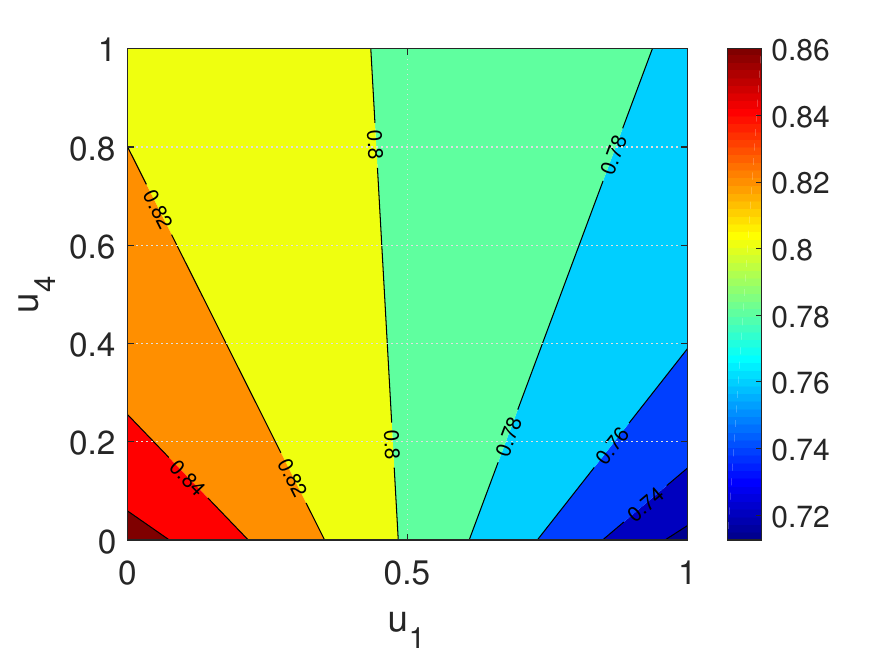}
\centering{(b)}
\end{minipage}
\centering
\caption{\centering Mesh and contour plots showing the combined influence 
of parameters on $\mathcal{ R}_e$. 
(a) $\mathcal{ R}_e$ in terms of $u_2$ and $u_4$. 
(b)  $\mathcal{ R}_e$ in terms of $u_1$ and $u_4$.}
\label{Fig12}
\end{figure}
\begin{figure}[H]
\begin{minipage}[b]{0.45\textwidth}
\includegraphics[height=5.0cm, width=7.5cm]{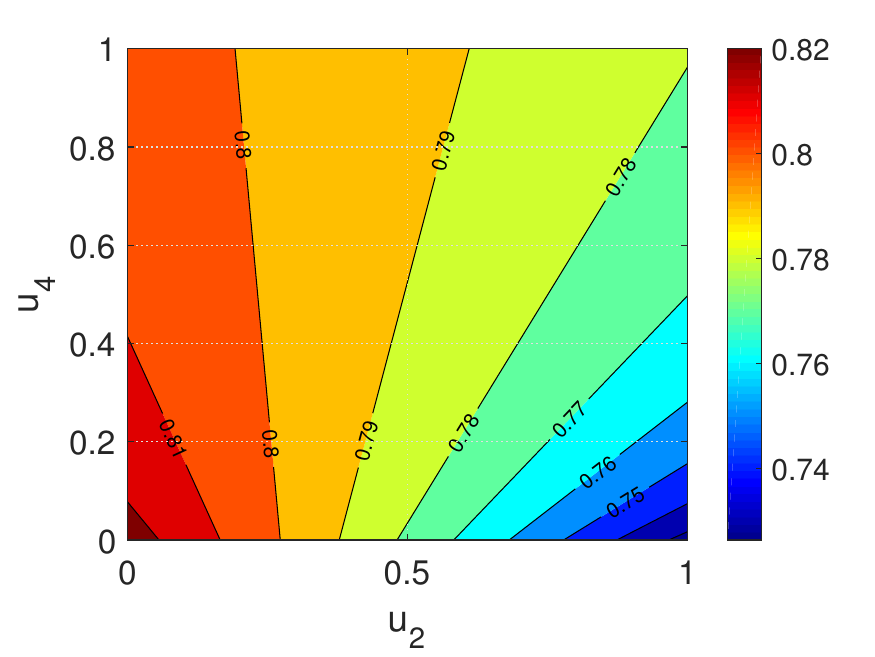}
\centering{(a)}
\end{minipage}
\begin{minipage}[b]{0.45\textwidth}
\includegraphics[height=5.0cm, width=7.5cm]{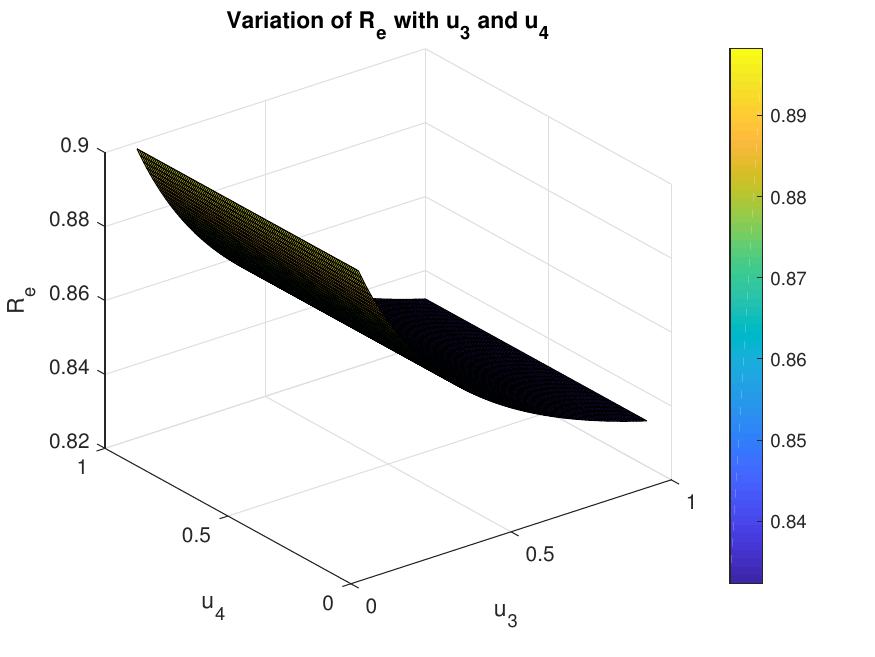}
\centering{(b)}
\end{minipage}
\centering
\caption{\centering  Mesh and contour plots showing the combined influence 
of parameters on  $\mathcal{ R}_e$. 
(a) $\mathcal{ R}_e$ in terms of $u_2$ and $u_4$. 
(b)  $\mathcal{ R}_e$ in terms of $u_3$ and $u_4$.}
\label{Fig13}
\end{figure}

The mesh and contour plot representing $\mathcal{R}e$ as a function 
of $\psi_1$ and $\psi_{2}$, which influence the effective reproduction 
number in the dynamics of rabies transmission, are shown in Figure~\ref{Fig10}. 
The results demonstrate that the effectiveness increases along with the values 
of $\psi_1$ and $\psi_{2}$. This indicates that increased rates of contact 
have a positive effect on the dynamics of rabies transmission in domestic dogs.
\begin{figure}[H]
\begin{minipage}[b]{0.45\textwidth}
\includegraphics[height=5.0cm, width=7.5cm]{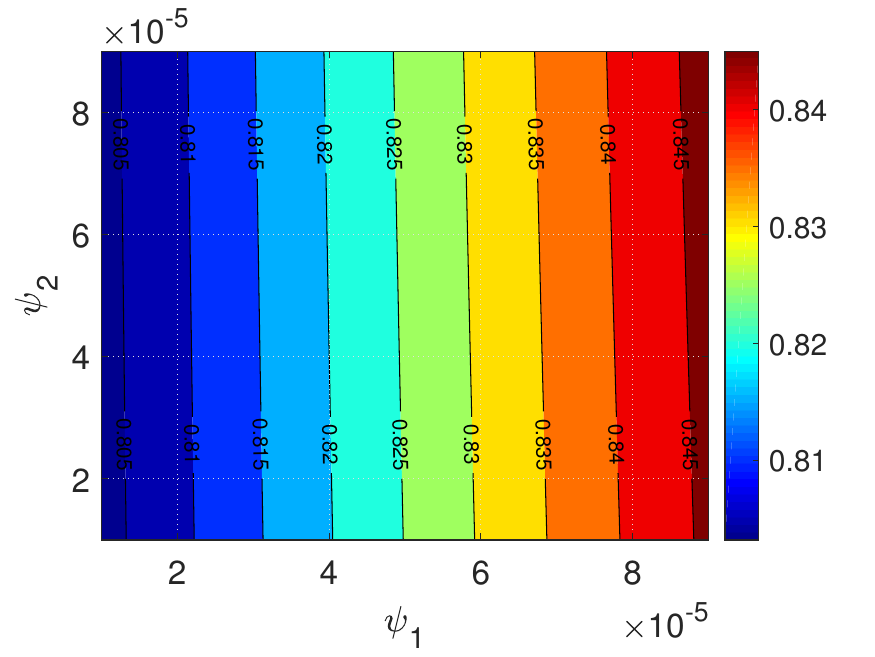}
\centering{(a)}
\end{minipage}
\begin{minipage}[b]{0.45\textwidth}
\includegraphics[height=5.0cm, width=7.5cm]{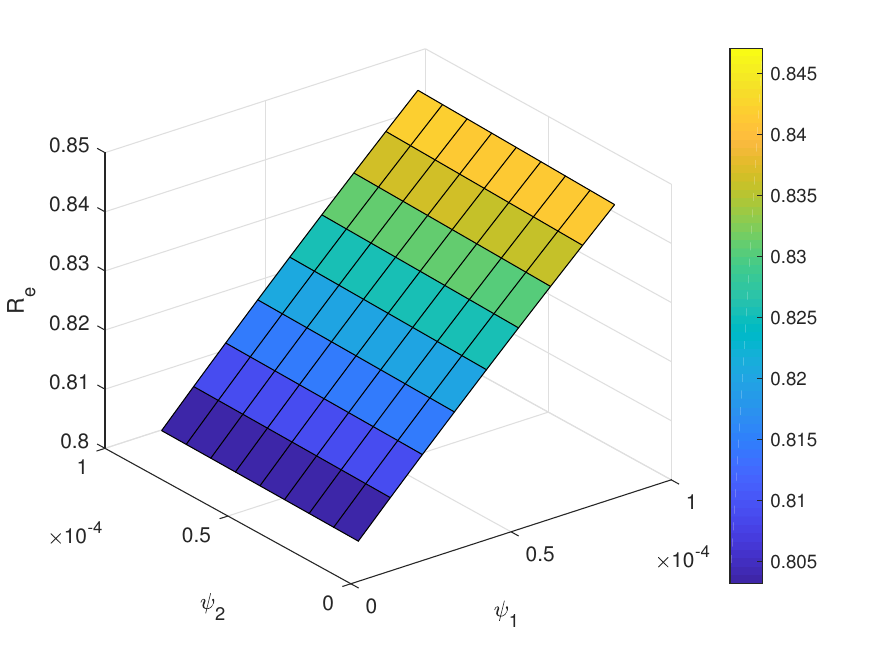}
\centering{(b)}
\end{minipage}
\centering 
\caption{\centering  Variations in $\mathcal{ R}_e$  
with respect to changes on (a) the rate that $S_D$ gets infection from 
$I_F$ and  $\psi_1$; (b) the rate that $S_D$ gets infection from 
$I_D$ and  $\psi_1$.}
\label{Fig10}
\end{figure}


\subsection{Impact of deterrence factors on the transmission of rabies in domestic dogs}

Figure~\ref{Fig2} demonstrates that an increase in deterrence factors leads 
to a reduction on the number of  exposed and  infected domestic dogs. 
This suggests that deterrence factors exert a significant influence 
on the transmission dynamics of rabies in domestic dogs.
\begin{figure}[H]
\includegraphics[scale=0.6]{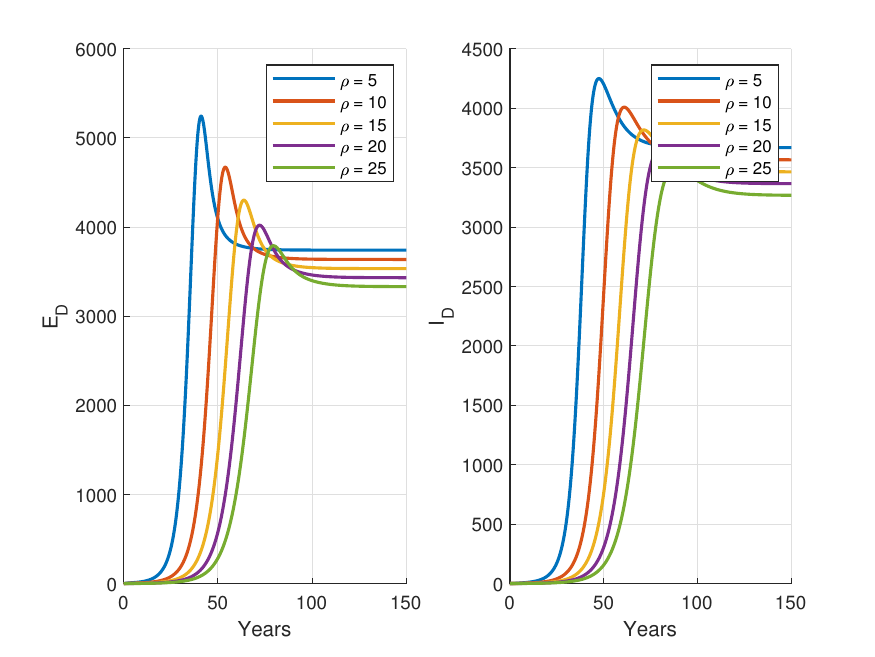}
\centering
\caption{\centering Impact of the deterrent factor $\rho$ 
on the transmission  dynamics of rabies in domestic dogs.}
\label{Fig2}
\end{figure}


\subsection{Numerical simulations} 
\label{sec:32}

Now we conduct numerical simulations involving an optimal control model 
to analyze the dynamics and management of rabies in humans, as well as in 
free-range and domestic dog populations. The optimal control solution is obtained by 
solving the optimality system, which consists of the adjoint system 
and the state system, according with Theorem~\ref{thm:PMP}. To solve such system,
the fourth-order Runge--Kutta iterative scheme method is employed, 
initializing controls over the simulated time.
For solving the adjoint equations, the backward fourth-order Runge--Kutta scheme is employed, 
leveraging existing solutions of the state equations in accordance with the transversality 
conditions \cite{yusuf2022optimal}. Furthermore, updates to the controls $\left(u_1\;,u_2\;,u_3\;,u_4\right)$ 
are made using a convex combination of the prior controls and values from the characterizations 
\cite{yusuf2022optimal}. This iterative process continues until the values of unknowns 
in subsequent iterations closely approximate those in the current iteration 
\cite{sagamiko2015optimal}. Our \textsf{Matlab} implementation follows the one  
discussed in \cite{MR4091761}. The  values  used  for  numerical   simulation   
are shown in Table~\ref{T2}. 

We incorporate four time dependent control variables that are: $u_1\left(t\right)$, 
which measures the effect of promoting good health practice and management 
(surveillance and monitoring) in reducing the possibility of human and dogs infection  
with rabies; $u_2\left(t\right)$, which represents the control effort due to vaccination 
of domestic dogs; $u_3\left(t\right)$, which measures the impact of educating 
communities about rabies transmission risks and protective measures to avoid 
exposure (public awareness and education); $u_4\left(t\right)$, which represents the 
treatment effort to individuals exposed to suspected rabid animals through bites 
or scratches (Post-Exposure Prophylaxis-PEP). Strategies for eliminating rabies 
in humans and dog populations have been ranked based on their effectiveness, 
starting from the least effective and advancing to the most effective as follows.


\subsubsection*{Strategy A: optimal control with all the controls $u_i$, $i=1,2,3,4$.}

Here we measure the impact of promoting good health practices and management, 
vaccinating domestic dogs, educating communities, and lastly, providing treatment for
those who have been in contact with a potentially rabid animal through bites or scratches 
(Post-Exposure Prophylaxis-PEP). This strategy aims to eradicate rabies 
within the community by focusing on four 
key interventions: promoting good health practices and management, denoted by $u_1(t)$; 
vaccinating domestic dogs, represented by $u_2(t)$; educating communities, 
captured by $u_3(t)$; and administering treatment to individuals exposed 
to potentially rabid animals through Post-Exposure Prophylaxis (PEP), 
modeled by $u_4(t)$. The results demonstrate that the strategy effectively eliminates 
rabies within five years (see Figures~\ref{C1H}, \ref{CT1D}, and \ref{CT1F}), 
with optimal utilization of $u_2(t)$ and $u_4(t)$ at 100\% during the first 16 and 6 years respectively, 
followed by a gradual reduction to non-zero. Meanwhile, $u_1(t)$ and $u_3(t)$ declined over time, 
reflecting the fact that vaccination does not confer permanent immunity, 
as illustrated in the control profile.
\begin{figure}[H]
\centering
\includegraphics[scale=0.8]{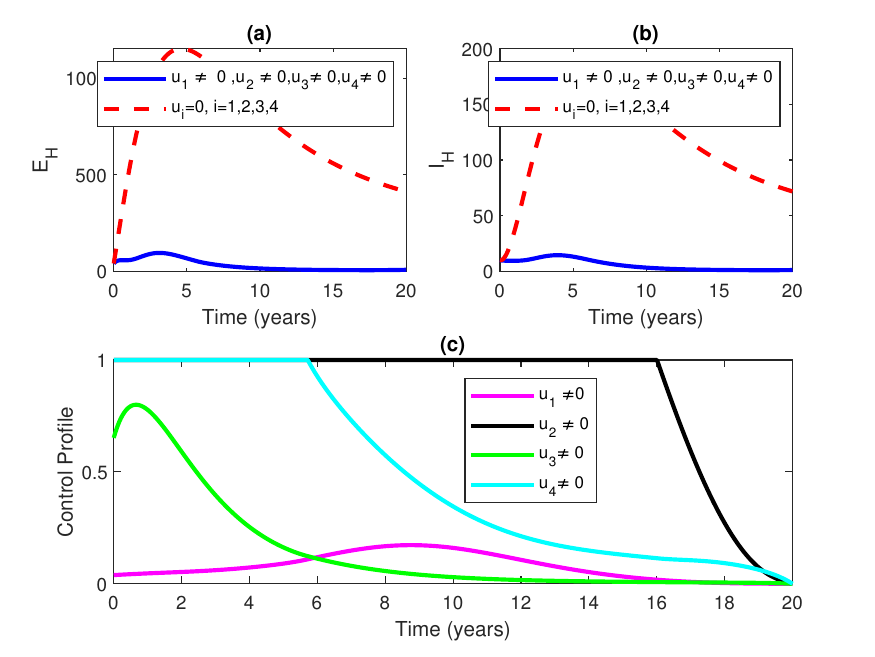}
\caption{\centering Impact of Strategy~A to human population.}
\label{C1H}
\end{figure}
\begin{figure}[H]
\centering
\includegraphics[scale=0.8]{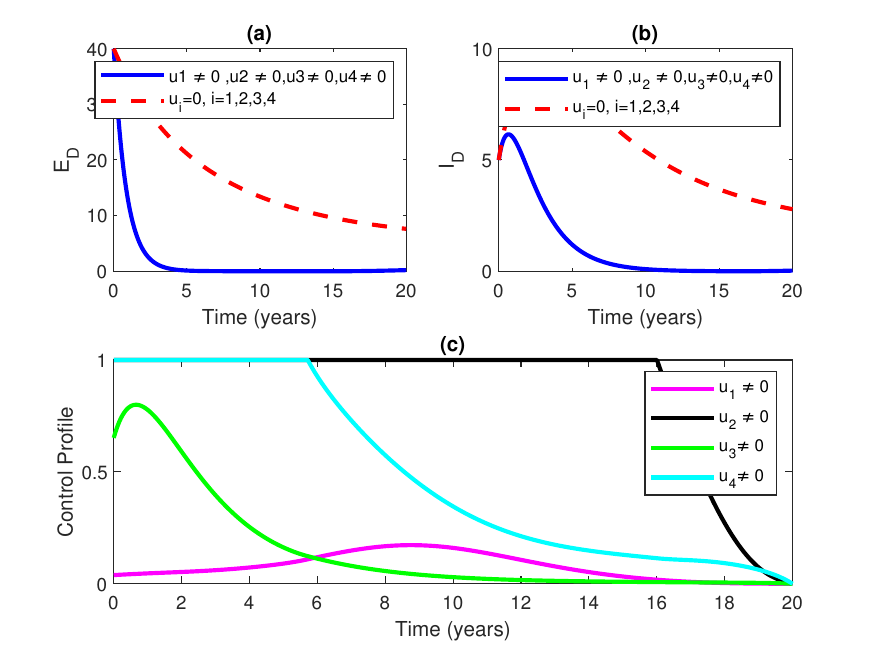}
\caption{\centering Impact of Strategy~A to domestic dogs population.}
\label{CT1D}
\end{figure}
\begin{figure}[H]
\centering
\includegraphics[scale=0.8]{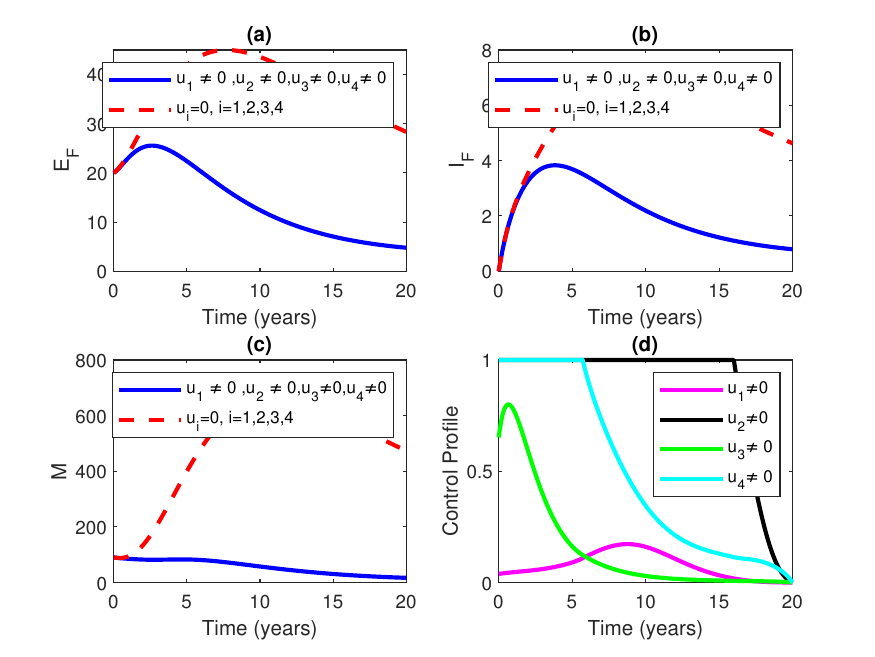}
\caption{\centering Impact of Strategy~A to free-range dogs and rabies virus.}
\label{CT1F}
\end{figure}


\subsubsection*{Strategy B: optimal control with controls $u_3$ and $u_4$.}

We assess the outcomes of public education and awareness campaigns and the availability 
of treatment for people exposed to  potential rabies via animal 
bites or scratches (Post-Exposure Prophylaxis-PEP).
The results are given in Figures~\ref{C3H}, \ref{CT3D} and \ref{CT3F}.
\begin{figure}[H]
\centering
\includegraphics[scale=0.8]{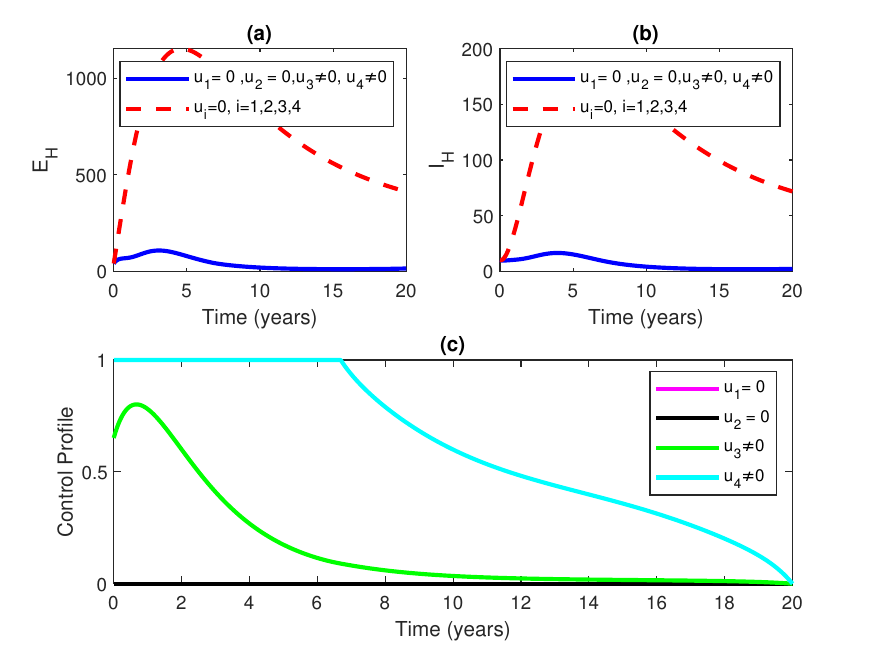}
\caption{Impact of Strategy~B to human population.}
\label{C3H}
\end{figure}
\begin{figure}[H]
\centering
\includegraphics[scale=0.8]{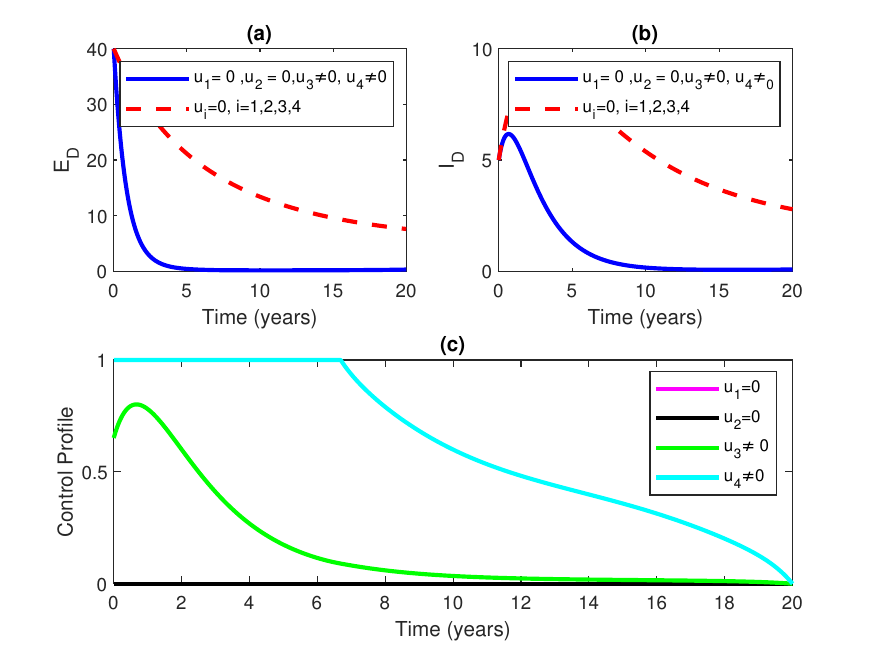}
\caption{Impact of Strategy~B to domestic dogs population.}
\label{CT3D}
\end{figure}
\begin{figure}[H]
\centering
\includegraphics[scale=0.8]{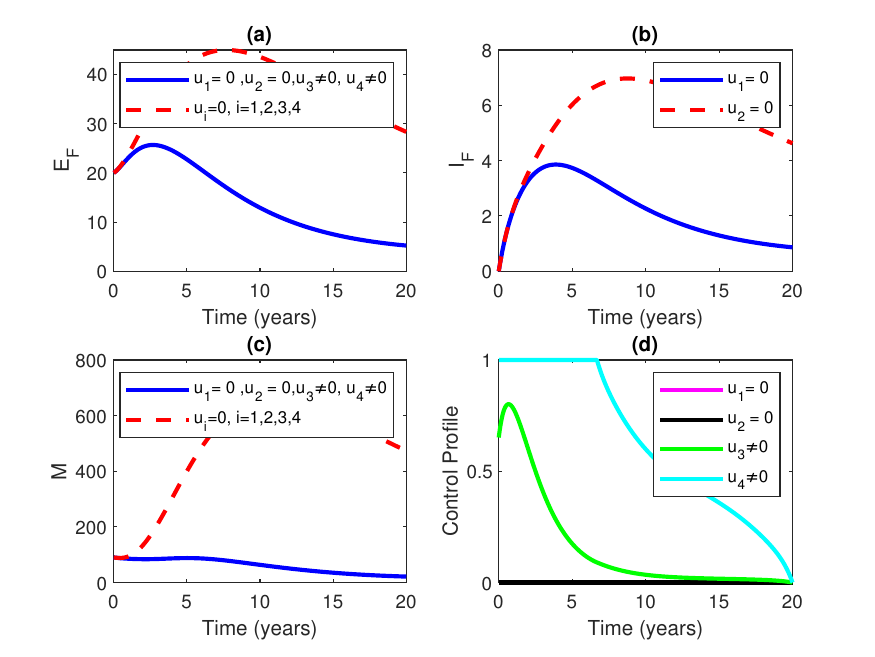}
\caption{Impact of Strategy~B to free-range dogs and rabies virus.}
\label{CT3F}
\end{figure}
  
  
\subsubsection*{Strategy C: optimal control using only control $u_4$.}

Now we concentrate on the impact of providing treatment  for bitten 
or scratched individuals by an animal suspected to be rabid  
(Post-Exposure Prophylaxis-PEP). Figures~\ref{C4H}, \ref{CT4D}, 
and \ref{CT4F} illustrate that the timely administration 
of PEP at 100\% effectiveness plays a pivotal role 
in reducing rabies mortality and curbing disease transmission 
within the first five years. By treating individuals immediately after exposure, 
PEP significantly reduces the risk of death and halts secondary transmission, 
particularly in high-risk areas. The results also reveal that PEP complements 
long-term rabies control measures, such as vaccination and health education.
While these strategies are vital, PEP offers immediate protection, 
ensuring effective containment of the disease even after exposure.
\begin{figure}[H]
\centering
\includegraphics[scale=0.8]{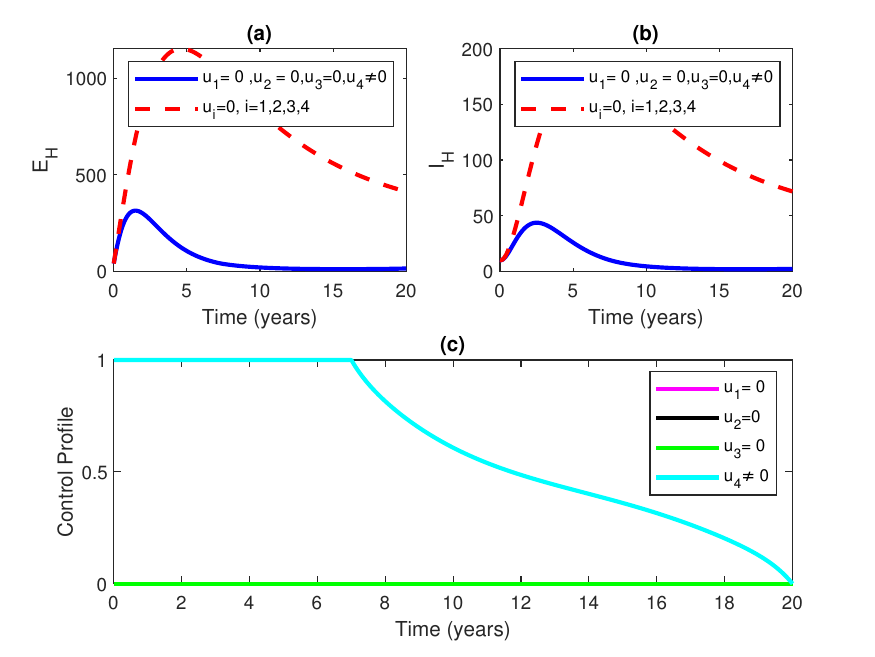}
 \caption{Impact of Strategy~C to human population.}
\label{C4H}
\end{figure}
\begin{figure}[H]
\centering
\includegraphics[scale=0.8]{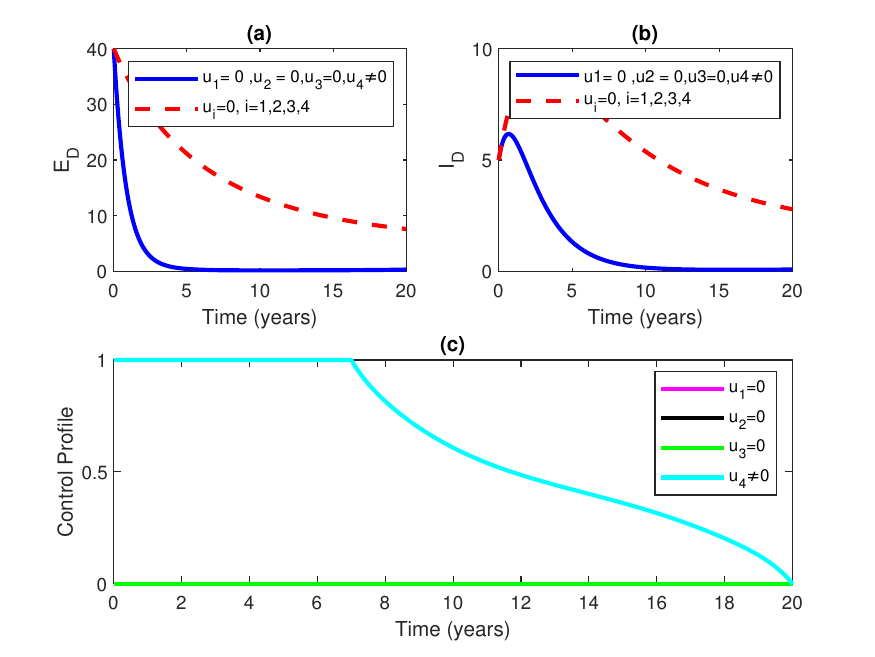}
\caption{Impact of Strategy~C to domestic dogs population.}
\label{CT4D}
\end{figure}
\begin{figure}[H]
\centering
\includegraphics[scale=0.8]{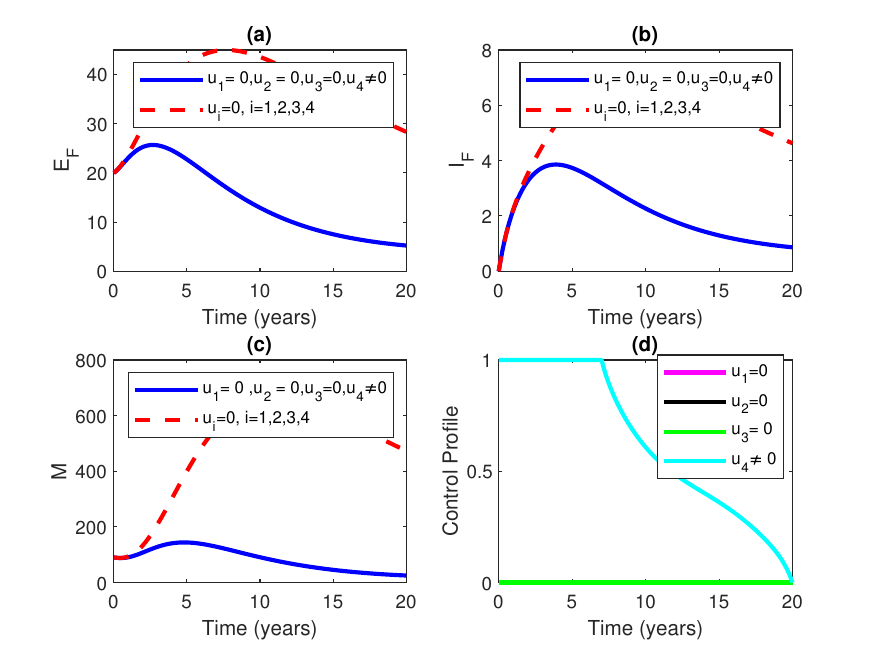}
\caption{Impact of Strategy~C to free-range dogs and rabies virus.}
\label{CT4F}
\end{figure}
    

\subsubsection*{Strategy D: optimal control with controls $u_1$ and $u_2$.}  

Finally, we consider the impact of promoting good health practices 
and management (surveillance and monitoring) and vaccinating on domestic dogs.
This strategy focuses on the combined impact of promoting good health practices 
and management, including surveillance and monitoring, represented by $u_1(t)$, 
along with the vaccination of domestic dogs, denoted by $u_2(t)$. 
The results indicate that the strategy is most effective when both control measures 
are applied at 100\%, as illustrated in Figures~\ref{CT6H}, \ref{CT6D}, and \ref{CT6F}. 
The effectiveness of the strategy was observed between the 4th and 14th years, 
followed by a decline after the reduction in control application. These findings 
suggest that promoting good health practices and vaccination can significantly 
reduce periodic rabies transmission. However, the data also highlights that vaccinating 
domestic dogs alone is not sufficient to fully prevent cyclical rabies transmission, 
primarily due to the movement of free-roaming dogs that may travel 
to rabies-free areas and initiate new outbreaks.
\begin{figure}[H]
\centering
\includegraphics[scale=0.8]{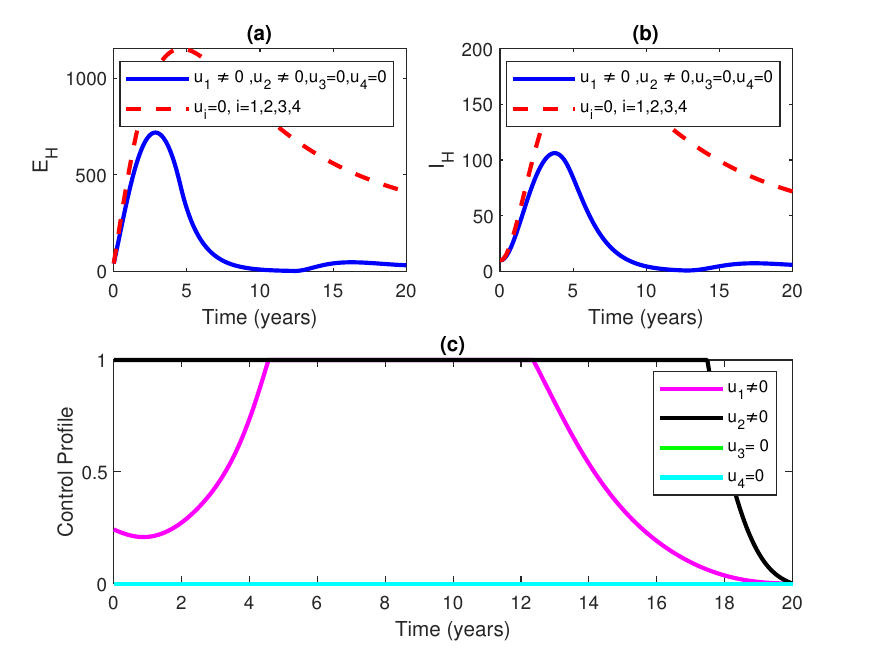}
\caption{\centering Impact of Strategy~D to human population.}
\label{CT6H}
\end{figure}
\begin{figure}[H]
\centering
\includegraphics[scale=0.8]{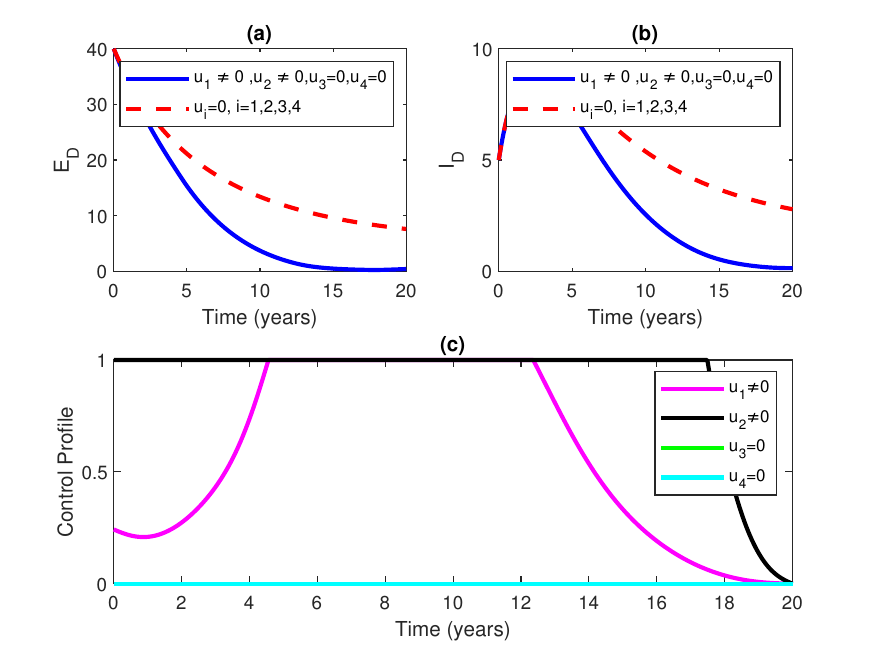}
\caption{\centering Impact of Strategy~D to domestic dogs population.}
\label{CT6D}
\end{figure}
\begin{figure}[H]
\centering
\includegraphics[scale=0.8]{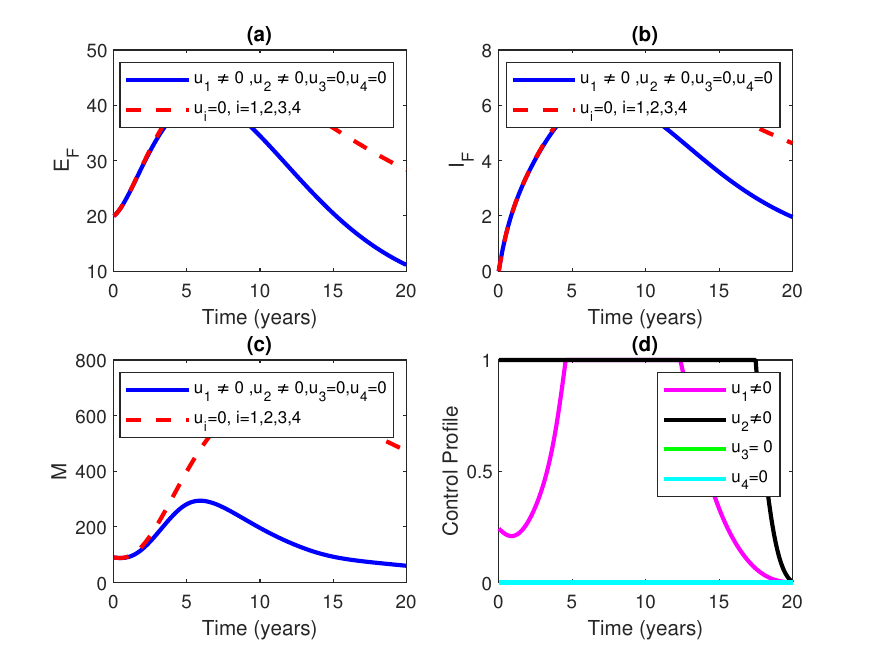}
\caption{Impact of Strategy~D to free-range dogs and rabies virus.}
\label{CT6F}
\end{figure}
   

\section{Conclusion}
\label{sec:disc:conc}

In this work, a new deterministic mathematical model for rabies disease, 
incorporating time-dependent control measures, was formulated and rigorously analyzed. 
The comparison between our effective reproduction rabies number $\mathcal{R}_e$ and 
the basic reproduction number $\mathcal{R}_0$ of \cite{charles2024mathematical},
underscores the importance of implementing robust control measures to reduce the effective 
reproduction number and control rabies transmission. The findings highlight the effectiveness  
of promoting good health practices, vaccinating domestic dogs, and providing treatment  
to exposed individuals. Among the control strategies, vaccination and PEP treatment 
stand out  as the most impactful measures for reducing the rabies spread. 
However, it is clear that relying on a single intervention may not be sufficient to fully 
eliminate rabies transmission. A combination of interventions, including vaccination, 
health education, and treatment, proves to be the most effective approach. Additionally, 
controlling the movement of free-roaming dogs and implementing deterrence measures further 
enhances the success of rabies elimination strategies. Ultimately, sustained control efforts, 
especially in regions with high rates of dog-to-dog and dog-to-human contact, are necessary 
to achieve long-term rabies eradication. A holistic approach that incorporates vaccination, 
education, treatment, and deterrence can significantly reduce the transmission dynamics 
and bring the effective reproduction number below the critical threshold 
needed to eliminate rabies.


\appendix

\section{Global stability of Endemic Equilibrium point $E^{*}$}
\label{AppA}

Follows the proof of the result asserted in Section~\ref{sec:2.2.4}.

\begin{theorem}
\label{thm:06}
The rabies disease  endemic  equilibrium point  of  model \eqref{eqn79} 
is globally asymptotically  stable if $\mathcal{R}_e>1$.
\end{theorem}

\begin{proof}
We use the Lyapunov function of the model system \eqref{eqn79} as described by 
\cite{charles2024mathematical,stephano2024significance,gumucs2025cost}. 
The Lyapunov function $\mathcal{H}$ is defined by
\begin{equation*}
\begin{aligned}
\mathcal{H}=\sum_{i=1}^{12} Q_{i}\left(y_{i}-y^{*}_{i}-y^{*}_{i}
\ln\left(\dfrac{y_{i}}{y^{*}_{i}}\right)\right),
\end{aligned}
\end{equation*}
where  $Q_{i}$ denotes an undetermined positive constant,  
$y_{i}$ represents the population of the  $i^{th}$  compartment, 
and $y^{*}_{i}$  indicates the endemic equilibrium point of model  
\eqref{eqn79}. It is apparent that the function  $\mathcal{H}$   
satisfies all the criteria for being a Lyapunov function:
\begin{itemize}
\item[(i)]  at the equilibrium point, $\mathcal{H}=0$ , thus 
$E^{*}\left(S^{*}_{H},\;E^{*}_{H},\;I^{*}_{H},\;R^{*}_{H},\;S^{*}_{F},\;
E^{*}_{F},\;I^{*}_{F},\;S^{*}_{D},\;E^{*}_{D},\;I^{*}_{D},\;R^{*}_{D},\;M\right)$;
\item[(ii)] $\mathcal{H}$ is positive for all other values of 
$S_{H},\;E_{H},\;I_{H},\;R_{H},\;S_{F},\;E_{F},\;I_{F},\;S_{D},\;E_{D},\;I_{D},\;R_{D}$ and $M$.
\end{itemize}
A Lyapunov function $\mathcal{H}$ of the model system \eqref{eqn79} is then defined as
\begin{equation}
\begin{aligned}
\mathcal{H} 
&= Q_{1}\left(S_H-S^{*}_{H}-S^{*}_{H}\ln\left(\frac{S_{H}}{S^{*}_{H}}\right)\right) 
+Q_{2}\left(E_H-E^{*}_{H}-E^{*}_{H}\ln\left(\frac{E_{H}}{E^{*}_{H}}\right)\right) \\
&\quad +Q_{3}\left(I_H-I^{*}_{H}-I^{*}_{H}\ln\left(\frac{I_{H}}{I^{*}_{H}}\right)\right)  
+Q_{4}\left(R_H-R^{*}_{H}-R^{*}_{H}\ln\left(\frac{R_{H}}{R^{*}_{H}}\right)\right) \\
&\quad +Q_{5}\left(S_F-S^{*}_{F}-S^{*}_{F}\ln\left(\frac{S_{F}}{S^{*}_{F}}\right)\right)  
+Q_{6}\left(E_F-S^{*}_{F}-E^{*}_{F}\ln\left(\frac{E_{F}}{E^{*}_{F}}\right)\right) \\
&\quad +Q_{7}\left(I_F-I^{*}_{F}-I^{*}_{F}\ln\left(\frac{I_{F}}{I^{*}_{F}}\right)\right)  
+Q_{8}\left(S_D-S^{*}_{D}-S^{*}_{D}\ln\left(\frac{S_{D}}{S^{*}_{D}}\right)\right) \\
&\quad +Q_{9}\left(E_D-E^{*}_{D}-E^{*}_{D}\ln\left(\frac{E_{D}}{E^{*}_{D}}\right)\right)  
+Q_{10}\left(I_D-I^{*}_{D}-I^{*}_{D}\ln\left(\frac{I_{D}}{I^{*}_{D}}\right)\right) \\
&\quad +Q_{11}\left(R_D-S^{*}_{D}-R^{*}_{D}\ln\left(\frac{R_{D}}{R^{*}_{D}}\right)\right)  
+Q_{12}\left(M-M^{*}-M^{*}\ln\left(\frac{M}{M^{*}}\right)\right).
\end{aligned}
\label{H}
\end{equation}
The  time derivative of equation $\left(\ref{H} \right)$  
leads to
\begin{equation}
\begin{aligned}
\dfrac{d\mathcal{H}}{dt} &= Q_{1}\left(1- \dfrac{S^{*}_{H}}{S_H}
\right)\dfrac{dS_H}{dt} +Q_{2}\left(1-\dfrac{E^{*}_{H}}{E_H}\right)
\dfrac{dE_H}{dt}  +Q_{3}\left(1-\dfrac{I^{*}_{H}}{I_H}\right) \dfrac{dI_H}{dt}\\
&\quad +Q_{4}\left(1-\dfrac{R^{*}_{H}}{R_H}\right) \dfrac{dR_H}{dt} 
+Q_{5}\left(1- \dfrac{S^{*}_{F}}{S_F}\right)\dfrac{dS_F}{dt}  
+Q_{6}\left(1-\dfrac{E^{*}_{F}}{E_F}\right)\dfrac{dE_F}{dt} \\
&\quad +Q_{7}\left(1-\dfrac{I^{*}_{F}}{I_F}\right) \dfrac{dI_F}{dt}  
+Q_{8}\left(1- \dfrac{S^{*}_{D}}{S_D}\right)\dfrac{dS_D}{dt}  
+Q_{9}\left(1-\dfrac{E^{*}_{D}}{E_D}\right)\dfrac{dE_D}{dt} \\
&\quad +Q_{10}\left(1-\dfrac{I^{*}_{D}}{I_D}\right) \dfrac{dI_D}{dt}  
+Q_{11}\left(1-\dfrac{R^{*}_{D}}{R_D}\right) \dfrac{dR_D}{dt}  
+Q_{12}\left(1-\dfrac{M^{*}}{M}\right) \dfrac{dM}{dt}.
\end{aligned}
\label{dH}
\end{equation}
Substituting  equation \eqref{eqn79} into  equation $\left(\ref{dH}\right)$ yields
\begin{equation}
\begin{aligned}
\dfrac{d\mathcal{H}}{dt} 
&= Q_{1}\left(1- \dfrac{S^{*}_{H}}{S_H}\right)
\left[\theta_{1}+\beta_{3}R_{H}-\mu_{1} S_{H}
-\left(1-\eta_{1}\right)f_{1} S_{H}\right]\\
&\quad +Q_{2}\left(1-\dfrac{E^{*}_{H}}{E_H}\right)\left[\left(1-\eta_{1}\right) 
f_{1}S_{H}- \eta_{12}E_{H}\right] +Q_{3}\left(1-\dfrac{I^{*}_{H}}{I_H}\right) 
\left[\beta_{1}E_{H}-\eta_{2} I_{H}\right]\\
&\quad  +Q_{4}\left(1-\dfrac{R^{*}_{H}}{R_H}\right) 
\left[\eta_{3} E_{H}-\eta_{4} R_{H}\right]
+Q_{5}\left(1- \dfrac{S^{*}_{F}}{S_F}\right)\left[\theta_{2}-f_2S_{F}-\mu_{2}S_{F}\right]\\
&\quad  +Q_{6}\left(1-\dfrac{E^{*}_{F}}{E_F}\right)\left[f_{2} S_{F}-\eta_{5}E_{F}\right] 
+Q_{7}\left(1-\dfrac{I^{*}_{F}}{I_F}\right) \left[\gamma E_{F}-\eta_{6}I_{F}\right]\\
&\quad   +Q_{8}\left(1- \dfrac{S^{*}_{D}}{S_D}\right)\left[\theta_{3}-\mu_{3}S_{D}
-\left(1-\eta_{7}\right)f_{3}S_{D}+\gamma_{3}R_{D}\right]\\
&\quad +Q_{9}\left(1-\dfrac{E^{*}_{D}}{E_D}\right)
\left[\left(1-\eta_{7}\right) f_{3}S_{D}-\eta_{8} E_{D}\right]  
+Q_{10}\left(1-\dfrac{I^{*}_{D}}{I_D}\right) \left[\gamma_{1}E_{D}
-\eta_{9} I_{D}\right] \\&\quad  +Q_{11}\left(1-\dfrac{R^{*}_{D}}{R_D}\right) 
\left[\eta_{10}E_{D}-\eta_{11}R_{D}\right] 
+Q_{12}\left(1-\dfrac{M^{*}}{M}\right) \left[f_{4}-\mu_4M	\right],
\end{aligned}
\label{dE}
\end{equation}
where
\begin{equation*}
\begin{aligned}
\eta_{1} &= u_{1} + u_{3},\; 
\eta_{2} = \sigma_{1} + \mu_{1},\;
f_{1} = \tau_{1} I_{F} + \tau_{2} I_{D} + \tau_{3} \lambda(M), \\
\eta_{3} &=\beta_{2}+u_4,\; \eta_{4}=\beta_{3}+\mu_{1},\;
f_2 =\kappa_{1}I_{F} + \kappa_{2}I_{D} + \kappa_{3} \lambda(M),\\
\eta_5&=\mu_{2}+\gamma,\;\eta_{6}=\mu_{2}+\sigma_{2},\;
f_{3}=\left(\frac{\psi_{1}I_{F}}{1+\rho_{1}} 
+ \frac{\psi_{2}I_{D}}{1+\rho_{2}} 
+ \frac{\psi_{3}}{1+\rho_{3}}\lambda(M)\right),\\ 
\eta_{7}&=u_1+u_2,\;\eta_{8}=\mu_{3}+\gamma_{1}+\gamma_{2}
+u_{4},\;\eta_{9}=\mu_{3}+\sigma_{3},\;\eta_{10}=\gamma_{2}+u_4\\
\eta_{11}&=\mu_{3}+\gamma_{3},\;
f_{4}=\nu_1I_H+\nu_2I_F+\nu_3I_D,\;
\lambda(M) = \frac{M}{M+C},\\
\eta_{12}&=\mu_{1}+\beta_{1}+\beta_{2}+u_{4}.
\end{aligned}
\end{equation*}
At the rabies disease  endemic equilibrium point $E^{*}$ 
equation $\left(\ref{dE}\right)$ gives
\begin{equation}
\begin{aligned}
\dfrac{d\mathcal{H}}{dt} 
&= Q_{1}\left(1- \dfrac{S^{*}_{H}}{S_H}\right)\left[\left(1-\eta_{1}\right)
f^{*}_{1} S^{*}_{H}+\mu_{1} S^{*}_{H}-\beta_{3}R^{*}_{H}
+\beta_{3}R_{H}-\mu_{1} S_{H}-\left(1-\eta_{1}\right)f_{1} S_{H}\right]\\
&\quad +Q_{2}\left(1-\dfrac{E^{*}_{H}}{E_H}\right)\left[\left(1-\eta_{1}\right) 
f_{1}S_{H}- \dfrac{\left(1-\eta_{1}\right)f^{*}_{1}S^{*}_{H}E_{H}}{E^{*}_{H}}\right] 
+Q_{3}\left(1-\dfrac{I^{*}_{H}}{I_H}\right) \left[\beta_{1}E_{H}
-\dfrac{\beta_{1}E^{*}_{H}I_{H}}{I^{*}_{H}} \right]\\
&\quad  +Q_{4}\left(1-\dfrac{R^{*}_{H}}{R_H}\right) \left[\eta_{3} 
E_{H}-\dfrac{\eta_{3}E^{*}_{H}R_{H}}{R^{*}_{H}}\right]
+Q_{5}\left(1- \dfrac{S^{*}_{F}}{S_F}\right)\left[f^{*}_2S^{*}_{F}
+\mu_{2}S^{*}_{F}-f_2S_{F}-\mu_{2}S_{F}\right]\\
&\quad  +Q_{6}\left(1-\dfrac{E^{*}_{F}}{E_F}\right)\left[f_{2} 
S_{F}-\dfrac{f^{*}_{2}S^{*}_{F} E_{F}}{E^{*}_{F}}\right] 
+Q_{7}\left(1-\dfrac{I^{*}_{F}}{I_F}\right) \left[
\gamma E_{F}-\dfrac{\gamma E^{*}_{F}I_{F}}{I^{*}_{F}}\right]\\
&\quad   +Q_{8}\left(1-\dfrac{S^{*}_{D}}{S_D}\right)\left[
\left(1-\eta_{7}\right)f^{*}_{3}S^{*}_{D}+\mu_{3}S^{*}_{D}
-\gamma_{3}R^{*}_{D}-\left(1-\eta_{7}\right)f_{3}S_{D}
-\mu_{3}S_{D}+\gamma_{3}R_{D}\right]\\
&\quad +Q_{9}\left(1-\dfrac{E^{*}_{D}}{E_D}\right)\left[
\left(1-\eta_{7}\right) f_{3}S_{D}-\dfrac{\left(1-\eta_{7}\right)
f^{*}_{3}S^{*}_{D}E_{D}}{E^{*}_{D}} \right]  
+Q_{10}\left(1-\dfrac{I^{*}_{D}}{I_D}\right) 
\left[\gamma_{1}E_{D}-\dfrac{\gamma_{1} E^{*}_{D}I_{D}}{I^{*}_{D}} \right]\\
&\quad  +Q_{11}\left(1-\dfrac{R^{*}_{D}}{R_D}\right) \left[\eta_{10}E_{D}
-\dfrac{\eta_{10}E^{*}_{D}R_{D}}{R^{*}_{D}}\right] 
+Q_{12}\left(1-\dfrac{M^{*}}{M}\right) \left[f_{4}-\dfrac{f^{*}_{4}M}{M^{*}}\right].
\end{aligned}
\label{L1}
\end{equation}
Expanding $\left(\ref{L1}\right)$, we obtain that
\begin{equation}
\begin{aligned}
\dfrac{d\mathcal{H}}{dt} 
&= -Q_{1}\mu_{1} S_{H}\left(1- \dfrac{S^{*}_{H}}{S_H}\right)^{2}
-Q_{5}\mu_{2} S_{F}\left(1- \dfrac{S^{*}_{F}}{S_F}\right)^{2}-Q_{8}
\mu_{3} S_{D}\left(1- \dfrac{S^{*}_{D}}{S_D}\right)^{2}\\
&\quad + \left(1-\eta_{1}\right)Q_{1}f_{1} S_{H}
\left(1- \dfrac{S^{*}_{H}}{S_H}\right)\left(1- 
\dfrac{f^{*}_{1}S^{*}_{H}}{f_{1}S_H}\right)
+Q_{1}\beta_{1}R_{H}\left(1- \dfrac{S^{*}_{H}}{S_H}\right)\left(1- \dfrac{R^{*}_{H}}{R_H}\right)\\
&\quad +Q_{2}\left(1-\eta_{1}\right) f_{1}S_{H}\left(1-\dfrac{E^{*}_{H}}{E_H}\right)
\left(1- \dfrac{f^{*}_{1}S^{*}_{H}E_{H}}{f_{1}S_{H}E^{*}_{H}}\right) 
+Q_{3}\beta_{1}E_{H}\left(1-\dfrac{I^{*}_{H}}{I_H}\right) 
\left(1-\dfrac{E^{*}_{H}I_{H}}{I^{*}_{H}E_{H}} \right)\\
&\quad  +Q_{4}\eta_{3} E_{H}\left(1-\dfrac{R^{*}_{H}}{R_H}\right) 
\left(1-\dfrac{E^{*}_{H}R_{H}}{R^{*}_{H}E_{H}}\right)
+Q_{5}f_{2}S_F\left(1- \dfrac{S^{*}_{F}}{S_F}\right)
\left(1- \dfrac{f^{*}_{2}S^{*}_{F}}{f_{2}S_F}\right)\\
&\quad  +Q_{6}f_{2} S_{F}\left(1-\dfrac{E^{*}_{F}}{E_F}\right)
\left(1-\dfrac{f^{*}_{2}S^{*}_{F} E_{F}}{E^{*}_{F}S_{F}f_{2}}\right) 
+Q_{7}\gamma E_{F}\left(1-\dfrac{I^{*}_{F}}{I_F}\right) 
\left(1-\dfrac{ E^{*}_{F}I_{F}}{I^{*}_{F}E_{F}}\right)\\
&\quad +\left(1-\eta_{7}\right)Q_{8}f_{3} S_{D}\left(1- \dfrac{S^{*}_{D}}{S_D}\right)
\left(1- \dfrac{f^{*}_{3}S^{*}_{D}}{f_{3}S_D}\right)
+Q_{8}\gamma_{3}R_{D}\left(1- \dfrac{S^{*}_{D}}{S_D}\right)
\left(1- \dfrac{R^{*}_{D}}{R_D}\right)\\
&\quad +Q_{9}\left(1-\eta_{7}\right) f_{3}S_{D}\left(1-\dfrac{E^{*}_{D}}{E_D}\right)
\left(1-\dfrac{f^{*}_{3}S^{*}_{D}E_{D}}{E^{*}_{D} f_{3}S_{D}} \right)  
+Q_{10}\gamma_{1}E_{D}\left(1-\dfrac{I^{*}_{D}}{I_D}\right) 
\left(1-\dfrac{ E^{*}_{D}I_{D}}{I^{*}_{D}E_{D}} \right) \\
&\quad  +Q_{11}\eta_{10}E_{D}\left(1-\dfrac{R^{*}_{D}}{R_D}\right) 
\left(1-\dfrac{E^{*}_{D}R_{D}}{R^{*}_{D}E_{D}}\right) 
+Q_{12}f_{4}\left(1-\dfrac{M^{*}}{M}\right) 
\left(1-\dfrac{f^{*}_{4}M}{M^{*}f_{4}}\right).
\end{aligned}
\label{dE**}
\end{equation}
To simplify equation $\left(\ref{dE**}\right)$, let
\begin{equation*}
\begin{aligned}
a = \dfrac{S_{H}}{S^{*}_{H}}, \;
b = \dfrac{E_{H}}{E^{*}_{H}}, \;
c = \dfrac{I_{H}}{I^{*}_{H}}, \;
d = \dfrac{R_{H}}{R^{*}_{H}}, \;
e = \dfrac{S_{F}}{S^{*}_{F}}, \;
f = \dfrac{E_{F}}{E^{*}_{F}}, \;
g = \dfrac{I_{F}}{I^{*}_{F}}, \;
h = \dfrac{S_{D}}{S^{*}_{D}}, \\
m = \dfrac{E_{D}}{E^{*}_{D}}, \;
n = \dfrac{I_{D}}{I^{*}_{D}}, \;
p = \dfrac{R_{D}}{R^{*}_{D}}, \;
q = \dfrac{M}{M^{*}}, \;
k = \dfrac{f_{1}}{f^{*}_{1}}, \;
l = \dfrac{f_{2}}{f^{*}_{2}}, \;
t = \dfrac{f_{3}}{f^{*}_{3}}, \;\text{and}
s = \dfrac{f_{4}}{f^{*}_{4}}.
\end{aligned}
\end{equation*}
Upon solving $\left(\ref{dE**}\right)$, 
the implication equation yields to
\begin{equation}
\begin{aligned}
\dfrac{d\mathcal{H}}{dt} 
&= -Q_{1}\mu_{1} S_{H}\left(1- \dfrac{1}{a}\right)^{2}-Q_{5}\mu_{2} 
S_{F}\left(1- \dfrac{1}{e}\right)^{2}-Q_{8}\mu_{3} S_{D}\left(1- \dfrac{1}{h}\right)^{2}\\
&\quad + \left(1-\eta_{1}\right)Q_{1}f_{1} S_{H}\left(1- \dfrac{1}{a}\right)
\left(1- \dfrac{1}{ak}\right)+Q_{1}\beta_{1}R_{H}\left(1- \dfrac{1}{a}\right)\left(1- \dfrac{1}{d}\right)\\
&\quad +Q_{2}\left(1-\eta_{1}\right) f_{1}S_{H}\left(1-\dfrac{1}{b}\right)\left(1- \dfrac{b}{ka}\right)
+Q_{3}\beta_{1}E_{H}\left(1-\dfrac{1}{c}\right) \left(1-\dfrac{c}{b} \right)\\
&\quad  +Q_{4}\eta_{3} E_{H}\left(1-\dfrac{1}{d}\right) \left(1-\dfrac{d}{b}\right)
+Q_{5}f_{2}S_F\left(1- \dfrac{1}{e}\right)\left(1- \dfrac{1}{el}\right)\\
&\quad  +Q_{6}f_{2} S_{F}\left(1-\dfrac{1}{f}\right)\left(1-\dfrac{f}{el}\right) 
+Q_{7}\gamma E_{F}\left(1-\dfrac{1}{g}\right) \left(1-\dfrac{g}{f}\right)\\
&\quad   +\left(1-\eta_{7}\right)Q_{8}f_{3} S_{D}\left(1- \dfrac{1}{h}\right)
\left(1- \dfrac{1}{ht}\right)+Q_{8}\gamma_{3}R_{D}
\left(1- \dfrac{1}{h}\right)\left(1- \dfrac{1}{p}\right)\\
&\quad +Q_{9}\left(1-\eta_{7}\right) f_{3}S_{D}
\left(1-\dfrac{1}{m}\right)\left(1-\dfrac{m}{ht} \right)  
+Q_{10}\gamma_{1}E_{D}\left(1-\dfrac{1}{n}\right) \left(1-\dfrac{n}{m} \right)\\
&\quad  +Q_{11}\eta_{10}E_{D}\left(1-\dfrac{1}{p}\right) \left(1-\dfrac{p}{m}\right) 
+Q_{12}f_{4}\left(1-\dfrac{1}{q}\right) \left(1-\dfrac{q}{s}\right).
\end{aligned}
\label{dEE}
\end{equation}
If $Q_{i}= 1$ for $i=1\le i\le 12$, whereas the coefficients of $ak$, $el$, 
and $ht$ are set to zero, then equation \eqref{dEE} becomes
\begin{equation}
\begin{aligned}
\dfrac{d\mathcal{H}}{dt} 
&= -\mu_{1} S_{H}\left(1- \dfrac{1}{a}\right)^{2}-\mu_{2} 
S_{F}\left(1- \dfrac{1}{e}\right)^{2}-\mu_{3} S_{D}\left(1- \dfrac{1}{h}\right)^{2}\\
&\quad +\beta_{1}R_{H}\left(1- \dfrac{1}{d}-\dfrac{1}{a}+\dfrac{1}{ad}\right) 
+\beta_{1}E_{H} \left(1-\dfrac{c}{b}-\dfrac{1}{c}+\dfrac{1}{b} \right)\\
&\quad  +\eta_{3} E_{H}\left(1-\dfrac{d}{b}-\dfrac{1}{d}+\dfrac{1}{b}\right) 
+\gamma E_{F}\left(1- \dfrac{g}{f}-\dfrac{1}{g}+\dfrac{1}{f}\right)\\
&\quad +\gamma_{3}R_{D}\left(1- \dfrac{1}{p}-\dfrac{1}{h}+\dfrac{1}{ph}\right)
+ \gamma_{1}E_{D}\left(1-\dfrac{n}{m}-\dfrac{1}{n}+\dfrac{1}{m}\right)\\
&\quad  +\eta_{10}E_{D}\left(1-\dfrac{p}{m}-\dfrac{1}{p}+\dfrac{1}{m}\right)  
+f_{4}\left(1-\dfrac{q}{s}-\dfrac{1}{q}+\dfrac{1}{s}\right). 
\end{aligned}
\label{dEE2}
\end{equation}
From equation \eqref{dEE2} we have  
\begin{equation}
\begin{aligned}
H_{1}= -\mu_{1} S_{H}\left(1- \dfrac{1}{a}\right)^{2}-\mu_{2} S_{F}
\left(1- \dfrac{1}{e}\right)^{2}-\mu_{3} S_{D}\left(1- \dfrac{1}{h}\right)^{2}
\end{aligned}
\label{H1}
\end{equation}
and 
\begin{equation}
\begin{aligned}
H_{2}&=\beta_{1}R_{H}\left(1- \dfrac{1}{d}-\dfrac{1}{a}+\dfrac{1}{ad}\right) 
+\beta_{1}E_{H} \left(1-\dfrac{c}{b}-\dfrac{1}{c}+\dfrac{1}{b} \right)\\
&\quad  +\eta_{3} E_{H}\left(1-\dfrac{d}{b}-\dfrac{1}{d}+\dfrac{1}{b}\right) 
+\gamma E_{F}\left(1- \dfrac{g}{f}-\dfrac{1}{g}+\dfrac{1}{f}\right)\\
&\quad  
+\gamma_{3}R_{D}\left(1- \dfrac{1}{p}-\dfrac{1}{h}+\dfrac{1}{ph}\right)
+ \gamma_{1}E_{D}\left(1-\dfrac{n}{m}-\dfrac{1}{n}+\dfrac{1}{m}\right)  \\
&\quad  +\eta_{10}E_{D}\left(1-\dfrac{p}{m}-\dfrac{1}{p}+\dfrac{1}{m}\right)  
+f_{4}\left(1-\dfrac{q}{s}-\dfrac{1}{q}+\dfrac{1}{s}\right). 
\end{aligned}
\label{H2}
\end{equation}

\begin{lemma}
\label{def1}	
If $\delta(y) = 1 - y + \ln y$, then $\delta(y) \leq 0$  where  
$1 - y \leq -\ln y$ if and only if $y > 0$ for all $y \in \mathbb{R}^{12+}$.
\end{lemma}

From equation $\left(\ref{H2}\right)$ we have
\begin{equation}
\begin{aligned}
1- \dfrac{1}{d}-\dfrac{1}{a}+\dfrac{1}{ad}
=\left(1-\dfrac{1}{d}\right)+\left(1-\dfrac{1}{a}\right)-\left(1-\dfrac{1}{ad}\right).
\end{aligned}
\label{def2}
\end{equation}
Equation $\left(\ref{def2}\right)$ is then written using Lemma~\ref{def1} 
and the concept of geometric mean as follows:
\begin{equation}
\begin{aligned}
\left(1-\dfrac{1}{d}\right)+\left(1-\dfrac{1}{a}\right)-\left(1-\dfrac{1}{ad}\right)
\le -\ln\left(\dfrac{1}{d}\right)-ln\left(\dfrac{1}{a}\right)+\ln\left(\dfrac{1}{ad}\right)\\
\le\ln\left(a\times d\times\dfrac{1}{ad} \right)=\ln\left(1\right)=0.
\end{aligned}
\label{def22}
\end{equation}
Following similar procedures in $\left(\ref{def22}\right)$, we get
\begin{equation}
\begin{aligned}
1-\dfrac{c}{b}-\dfrac{1}{c}+\dfrac{1}{b}\le 0,\;1-
\dfrac{p}{m}-\dfrac{1}{p}+\dfrac{1}{m}\le 0,\;
1-\dfrac{d}{b}-\dfrac{1}{d}+\dfrac{1}{b}\le 0.
\end{aligned}
\end{equation}
Therefore, it follows that $\dfrac{d\mathcal{H}}{dt}\le0$ 
and $\dfrac{d\mathcal{H}}{dt}=0$ only at the endemic equilibrium point 
$\left(E^{*}\right)$. Hence, using Lasalle's extension to Lyapunov's  method, 
the limit set of each solution is obtained in the largest invariant  
set for which $S^{*}_{H}=S_{H}$,$E^{*}_{H}=E_{H}$, $I^{*}_{H}=I_{H}$, 
$R^{*}_{H}=R_{H}$, $S^{*}_{F}=S_{F}$, $E^{*}_{F}=E_{F}$, $S^{*}_{D}=S_{D}$,  
$E^{*}_{D}=E_{D}$,$I^{*}_{D}=I_{D}$, $R^{*}_{D}=R_{D}$, $M^{*}=M$,  
which is the singleton $\left\{E^{*}\right\}$. Hence, the rabies disease endemic  
equilibrium point $\left(E^{*}\right)$  of the model system  \eqref{eqn79} 
is globally  asymptotically  stable  on a given set  
when  $\mathcal{ R}_{e}>1$.  
\end{proof}


\section*{Data Availability}

The data supporting this study were sourced from existing literature.


\section*{Funding Statement}

Torres is supported by the Portuguese Foundation 
for Science and Technology (FCT) and CIDMA, 
project UID/04106/2025.


\section*{Conflicts of Interest}

There are no known financial or personal conflicts of interest disclosed 
by the authors that could have influenced the outcome of this study.


\section*{Acknowledgements}

We would like to acknowledge The Nelson Mandela African Institution 
of Science and Technology (NM-AIST) and the College of Business Education 
(CBE) for providing a conducive environment during the writing of this manuscript.



\end{document}